\documentclass[12pt,reqno,dvipsnames]{amsart}
  \usepackage{latexsym} 
  \usepackage{mathtools,thmtools}
  \usepackage[all]{xy}
  \usepackage{amsfonts,amssymb,amsmath,amsthm,mathrsfs}
  \usepackage{pifont}  
  \usepackage{enumerate}
  \usepackage{dcolumn}
	\usepackage[inline]{enumitem}
	\usepackage{etoolbox}

  \usepackage{hyperref} 
\newcolumntype{2}{D{.}{}{2.0}}
  \xyoption{2cell}
\usepackage[capitalise,noabbrev]{cleveref} 	
\makeatletter
\patchcmd{\@setaddresses}{\indent}{\noindent}{}{}
\patchcmd{\@setaddresses}{\indent}{\noindent}{}{}
\patchcmd{\@setaddresses}{\indent}{\noindent}{}{}
\patchcmd{\@setaddresses}{\indent}{\noindent}{}{}
\makeatother

 \usepackage{tikz}
\usetikzlibrary{arrows}

   \def\<{{\langle}} 
  \def\>{{\rangle}}

  \def\note#1{{}}

  \def\note#1{}

  \def\beq{\begin{equation}} 
  \def\eeq{\end{equation}}

  \def\id{\mathrm{id}} 
   
  \newcommand{\Ima}{\mathrm{Im}}

 \DeclareMathOperator{\twoheaddownarrow}{\raisebox{10pt}{\rotatebox{-90}{$\twoheadrightarrow$}}}

  \newcounter{zlist} 
  \newenvironment{zlist}{\begin{list}{(\arabic{zlist})}{ 
  \usecounter{zlist}\leftmargin2.5em\labelwidth2em\labelsep0.5em 
  \topsep0.6ex
  \parsep0.3ex plus0.2ex minus0.1ex}}{\end{list}}

  \newcounter{blist}

  \newcounter{rlist}

\def\stac#1{\raise-.2cm\hbox{$\stackrel{\displaystyle\otimes}{\scriptscriptstyle{#1}}$}}

\def\cten#1{\raise-.2cm\hbox{$\stackrel{\displaystyle\reallywidehat{\otimes}}
{\scriptscriptstyle{#1}}$}}

  \def\Label#1{\label{#1}\ifmmode\llap{[#1] }\else 
  \marginpar{\smash{\hbox{\tiny [#1]}}}\fi} 
  \def\Label{\label}

  \newtheorem{proposition}{Proposition}[section]
  \newtheorem{lemma}[proposition]{Lemma} 
  \newtheorem{corollary}[proposition]{Corollary} 
  \newtheorem{theorem}[proposition]{Theorem}

\theoremstyle{definition} 
  \declaretheorem[name=Definition,qed={$\lozenge$},sibling=proposition]{definition}

  \declaretheorem[name=Example,qed={$\bigtriangledown$},sibling=proposition]{example}

 \declaretheorem[name=Construction,qed={$\bullet$},sibling=proposition]{construction}

  \theoremstyle{remark} 
  \declaretheorem[name=Remark,qed={$\triangle$},sibling=proposition]{remark}

  \newcounter{c} 
  \renewcommand{\[}{\setcounter{c}{1}$$} 
  \newcommand{\etyk}[1]{\vspace{-7.4mm}$$\begin{equation}\Label{#1} 
  \addtocounter{c}{1}} 
  \renewcommand{\]}{\ifnum \value{c}=1 $$\else \end{equation}\fi} 
   \numberwithin{equation}{section}

\newcommand{\Ah}{\mathbf{Ah}}

\newcommand{\Ab}{\mathbf{Ab}}

\newcommand{\Prod}{\mathrm{ Prod}}

\newcommand{\Mod}{\mbox{-}\mathbf{ Mod}}
\newcommand{\HMod}{\mbox{-} \mathbf{ HMod}}

\newcommand{\heap}{{\bf Hp}}

\def\NN{{\mathbb N}}

\def\ZZ{{\mathbb Z}}

\def\GGG{{\mathscr G}}

\def\MMM{{\mathscr M}}

\newcommand{\eE}{\mathrm{E}}

\newcommand{\gG}{\mathrm{G}}
\newcommand{\hH}{\mathrm{H}}

\newcommand{\rR}{\mathrm{R}}

\newcommand{\tT}{\mathrm{T}}

\newcommand{\Cc}{\mathcal{C}}

\newcommand{\Gg}{\mathcal{G}}
\newcommand{\Hh}{\mathcal{H}}

\def\*C{{}^*\hspace*{-1pt}{\Cc}}

\def\text#1{{\rm {\rm #1}}}

\def\ol{\overline}

\def\Set{\mathbf{Set}}

 \def\k{\mathbf{k}}

 \def\1{\mathbf{1}}

      \def\tr{\mathbf{Tr}}

\def\id{\mathrm{id}}

\def\grp {\mathbf{Grp}}

\def\trs {\mathbf{Trs}}
\def\lto{\longmapsto}
\def\lra{\longrightarrow}

\def\1\mathbf{1}

\def\|#1{\overline{#1}}

\def\k{\Bbbk}

\newcommand{\op}[1]{{#1}^{\mathrm{op}}}
\newcommand{\GMod}{\mbox{-}\mathbf{Mod_\bullet}}
\def\tra#1#2#3{#1\triangleright_#2 #3}

\usepackage{scalerel,stackengine}
\stackMath
\newcommand\reallywidehat[1]{%
\savestack{\tmpbox}{\stretchto{%
  \scaleto{%
    \scalerel*[\widthof{\ensuremath{#1}}]{\kern.1pt\mathchar"0362\kern.1pt}%
    {\rule{0ex}{\textheight}}
  }{\textheight}%
}{2.4ex}}%
\stackon[-6.9pt]{#1}{\tmpbox}%
}
\begin{document}
\baselineskip=14.4pt
\title{Heaps of modules: Categorical aspects}

\author{Simion Breaz}

\address{"Babe\c s-Bolyai" University, Faculty of Mathematics and Computer Science, Str. Mihail Kog\u alniceanu 1, 400084, Cluj-Napoca, Romania}
\email{bodo@math.ubbcluj.ro}

\author{Tomasz Brzezi\'nski}

\address{
Department of Mathematics, Swansea University, 
Fabian Way,
  Swansea SA1 8EN, U.K.\ \newline 
Faculty of Mathematics, University of Bia{\l}ystok, K.\ Cio{\l}kowskiego  1M,
15-245 Bia\-{\l}ys\-tok, Poland}

\email{T.Brzezinski@swansea.ac.uk}

\author{Bernard Rybo\l owicz}
\address{Department of Mathematics, Heriot-Watt University, Edinburgh EH14 4AS,
\newline and Maxwell Institute for Mathematical Sciences, Edinburgh (UK)}
\urladdr{https://sites.google.com/view/bernardrybolowicz/}
\email{B.Rybolowicz@hw.ac.uk}

\author{Paolo Saracco}

\address{
D\'epartement de Math\'ematique, Universit\'e Libre de Bruxelles, 
Bd du Triomphe, B-1050 Brussels, Belgium.}

\urladdr{sites.google.com/view/paolo-saracco}
\urladdr{paolo.saracco.web.ulb.be}

\email{paolo.saracco@ulb.be}



\begin{abstract}
Connections between heaps of modules and (affine) modules over rings are explored. This leads to explicit, often constructive, descriptions of some categorical constructions and properties that are implicit in universal algebra and algebraic theories. In particular, it is shown that the category of groups with a compatible action of a truss $T$ (also called pointed $T$-modules) is isomorphic to the category of modules over the ring $\rR(T)$ universally associated to the truss. This is widely used in the explicit description of free objects. Next, it is proven that the category of heaps of modules over $T$ is isomorphic to the category of affine modules over $\rR(T)$ and, in order to make the picture complete, that (in the unital case) these are in turn equivalent to a specific subcategory of the slice category of pointed $T$-modules over $\rR(T)$. These correspondences and properties are then used to describe explicitly various (co)limits and to compare short exact sequences in the Barr-exact category of heaps of $T$-modules with short exact sequences as defined previously.
\end{abstract}    
\date\today
\maketitle
\setcounter{tocdepth}{1}
\tableofcontents


\section{Introduction}

In the 1920s, H. Pr\" ufer \cite{Pruf24}, R. Baer \cite{Baer29} and A.\ K.\ Su\v skevi\v c \cite{Sus37} introduced a novel algebraic structure called \emph{heap}. A heap is a set with a ternary operation limited by certain conditions, called associativity and Mal’cev's identities (see \cref{ssec:heaps}). It turns out that there is a deep connection between groups and heaps: heaps correspond to free transitive actions of groups on sets, which allows us to depart from the choice of a neutral element and focus on the set as an affine version of the group itself. 

Nearly 100 years later, trusses were introduced in \cite{Brz:tru} as structures describing two different distributive laws: the well-known ring distributivity and the one coming from the recently introduced braces, which are gaining popularity due to their role in the study of the set-theoretic solutions of the Yang-Baxter equation. The brace distributive law appeared earlier in the context of  quasi-rings of radical rings; see \cite{Kro68}. It turns out that both these structures, rings and braces, can be described elegantly by switching the group structure to a heap structure. This leads to the definition of a truss, which is a set $T$ with a ternary operation $[-,-,-]$ and a binary multiplication $\cdot$ satisfying the conditions of \cref{ssec:trusses}, the crucial one being the generalisation of ring and brace distributivity:
\[
a\cdot[b,c,d]=[a\cdot b,a\cdot c,a\cdot d] \qquad \text{and} \qquad [b,c,d]\cdot a=[b\cdot a,c\cdot a,d\cdot a],
\]
for all $a,b,c,d\in T$. Thanks to this, we can jointly approach brace and ring theory.

Every truss $T$ is a congruence class of a ring $\rR(T)$, the universal extension of $T$ into a ring (see \cite{ABR}). Trusses, even though similar to rings, differ significantly as the category of trusses has no zero object. 
Having the structure of a truss, it is natural to ask: what is the theory of modules over trusses? As one can expect, a module $M$ over a truss $T$ is a truss homomorphism from the truss $T$ to the endomorphisms of the abelian heap $M$, analogously to the case of modules over rings. In the previous paper \cite{BrBrRySa}, we have shown that with every module $M$ over a truss $T$, one can associate an affine version of the module $M$, denoted by $(M,\triangleright)$ and called a heap of $T$-modules. In the case of modules over a ring, we acquire affine modules in the sense of ring theory (see e.g.\ \cite{Ostermann-Schmidt}). Surprisingly, two non-isomorphic modules over a truss can give rise to the same affine structure. Moreover, every heap of $T$-modules is an affine version of a $T$-module coming from the endomorphisms of a group instead of a heap. In the present text, such $T$-modules will be called pointed $T$-modules to underline the fact that $T$ acts on a group, that is, a heap with a chosen point. 
Modules and heaps of modules over commutative trusses are modes in the sense of \cite{RomanowskaSmith} or, in other words, idempotent entropic algebras. In that case, some results in this paper can also be seen as non-commutative extensions of their modal analogues.

In this paper we extend the exploration of links between  $T$-modules, modules over rings and pointed $T$-modules from \cite{BrBrRySa} and present a fundamental new observation that every affine structure over a truss $T$ reduces to the affine structure over a particular, universally constructed ring $\rR(T)$. The paper is  arranged  as follows.

\cref{sec:pre} contains preliminaries in which we recall all the necessary definitions and facts on heaps, trusses, modules over trusses and heaps of modules. 

In \cref{Tam}, we show that categories of pointed $T$-modules and $\rR(T)$-modules are isomorphic (\cref{thm:dummyP}). Due to this, we can construct all limits and colimits of pointed $T$-modules from $\rR(T)$-modules  and {\it vice versa}. For example, in \cref{prop:freeTgroup} we use this  isomorphism of categories to describe the free functor for pointed $T$-modules. We conclude this section with \cref{thm:TheapRTaff} which states that the categories of heaps of $T$-modules and affine $\rR(T)$-modules are isomorphic.

\cref{HOM:CA} is devoted to categorical aspects of heaps of $T$-modules. First, we note that all limits and colimits exist as heaps of $T$-modules, heaps, and $T$-modules are varieties of algebras in the sense of universal algebra. Then we briefly describe equalizers, products and pullback of heaps of $T$-modules. After that, we proceed to colimits. We present explicit constructions of quotients (\cref{rem:quotients}), coequalizers (\cref{prop:coeq} and \cref{con:coeq}), coproducts (\cref{con:coprod}) and pushouts (\cref{ex:push}) of heaps of $T$-modules or, in view of \cref{thm:TheapRTaff}, affine $\rR(T)$-modules. We conclude the part on colimits by presenting \cref{thm:slices}, which states that the category of heaps of $T$-modules is equivalent to the slice category of pointed $T$-modules projecting onto $\rR(T)$.

In the last part, \cref{sec:five}, we propose the definition of exactness of short sequences for heaps. In \cref{Prop:exofheap} we show that,
for every exact sequence of non-empty heaps, there exists a choice of elements such that by fixing those elements in the heaps of $T$-modules, we will acquire an exact sequence of modules over the ring $\rR(T)$. Then in \cref{ssec:Barr}, by taking advantage of the fact that the category of heaps of $T$-modules is exact, we define exact sequences in the sense of Barr and we conclude the paper with \cref{thm:exex}, relating every Barr exact sequence of heaps of $T$-modules to a particular short exact sequence of heaps of $T$-modules and conversely.

\section{Preliminaries}\label{sec:pre}

We begin by collecting the basics of heaps and trusses which will be necessary in the sections to come.


\subsection{Heaps and their morphisms}\label{ssec:heaps}

A {\em heap} is a set $H$ with a ternary operation 
\[[-,-,-]:H\times H\times H\to H\] 
such that for all $a,b,c,d,e\in H$ the following axioms hold:
\begin{enumerate}[label=(H\arabic*)]
\item $[a,b,[c,d,e]] =[[a,b,c],d,e]$ \quad (Associativity),
\item $[a,b,b]=[b,b,a]=a$ \qquad\qquad (Mal'cev identities).
\end{enumerate}
Moreover, if $[a,b,c]=[c,b,a]$, for all $a,b,c\in H$, then $H$ is called an {\em abelian heap}.
A \emph{sub-heap} of a heap $H$ is a subset $S$ closed under the ternary operation. 

A {\em morphism of heaps} is a function $f \colon H\to H'$, between heaps which is compatible with the ternary operations, i.e., for all $a,b,c\in H$, $f([a,b,c])=[f(a),f(b),f(c)]$. We denote by $\heap$ the category of heaps and their morphisms and by $\Ah$ the full subcategory of abelian heaps and their morphisms.

\begin{remark}[{\cite[\S2.4]{Brz:par}}]\label{rem:subheaprel}
    Let $S \subseteq H$ be a sub-heap of an abelian heap $H$. Denote by $\sim_S ~ \subseteq H \times H$ the equivalence relation:
    \[a \sim_S b \qquad \iff \qquad [a,b,s] \in S, \quad \forall\, s\in S,\]
    called the {\em sub-heap relation}. It is a congruence over $H$: if $a_i \sim_S b_i$ for $i=1,2,3$, then $\left[a_1,a_2,a_3\right] \sim_S \left[b_1,b_2,b_3\right]$. Therefore, the quotient $H/S \coloneqq H/{\sim_S}$ is an abelian heap and the canonical projection $H \to H/S$ is a morphism of heaps. Moreover, any congruence over an abelian heap is a sub-heap relation.
\end{remark}

    A family of morphisms that plays a cardinal role in the theory of heaps consists of  \emph{translation automorphisms}, defined for all $a,b\in H$ by the formula
    \begin{equation}\label{eq:trans}
    \tau_a^b\colon H\lra H, \qquad x\lto [x,a,b].
    \end{equation}
    The set of all translation automorphisms together with the identity of $H$ is denoted by $\tr(H)$. It is a subgroup of the automorphism group of $H$, which is called the \emph{translation group} of $H$. In view of \eqref{eq:trans}, for any morphism of heaps $f:H\to H'$, the map
    \begin{equation}\label{trans.funct}
    \tr(f)\colon \tr(H)\lra \tr(H'), \qquad \tau_a^b\lto \tau_{f(a)}^{f(b)},
    \end{equation}
    is a homomorphism of groups. This gives a functor $\tr\colon \heap\lra \grp$ that restricts to a functor $\Ah\lra\Ab$, that we denote by $\tr$ again (see \cite[\S2.1]{BrBrRySa}).

\subsection{Heaps and groups}

With every group $(G,\cdot,e)$ we can associate a heap $\hH(G) = (G,[-,-,-])$ where $[x,y,z] = xy^{-1}z$ for all $x,y,z\in G$; every morphism of groups is automatically a morphism of heaps hence this assignment is functorial, leading to the functor $\hH \colon \grp \to \heap$. In the opposite direction, with every non-empty heap $H$ and $e\in H$, we can associate a group $\gG(H;e)=(H,[-,e,-])$, where $[-,e,-]$ is a binary operation acquired by fixing the middle variable in the ternary operation. The group $\gG(H;e)$ is called the {\em  $e$-retract of the heap $H$}. 
Every morphism of non-empty heaps $f \colon H \to H'$ yields a morphism of groups
\[\gG(f) \colon \gG(H;e_H) \lra \gG(H';e_{H'}), \qquad \gG(f) \coloneqq \tau_{f(e_H)}^{e_{H'}} \circ f,\]
Notice that $\hH(\gG(H;e_H)) = H$ for every non-empty heap $H$ and that $\gG(\hH(G);e_G) \cong G$ for every group $G$.

\begin{remark}\label{rem-heap-morphism}
Let $f:H\to H'$ be a morphism of heaps. For every group operation associated with $H$ there exists a group operation associated with $H'$ such that $f$ is a morphism of groups with respect to these operations.
\end{remark}

\begin{remark}[{\cite[\S2.1]{BrBrRySa}}]
    Let $H$ be a non-empty heap and $e \in H$. The assignment
    \[\gG(H;e) \lra \op{\tr(H)}, \qquad x \lto \tau_e^x\]
    is an isomorphism of groups with inverse
    \begin{equation*}
    \op{\tr(H)} \lra \gG(H;e), \qquad \tau_x^y \lto \tau_x^y(e).\qedhere
    \end{equation*}
\end{remark}

For the convenience of the interested reader, we record a third realisation of the group associated with a non-empty heap (see, e.g., \cite[\S6.3]{RomanowskaSmith}). On the Cartesian product $H\times H$, define the equivalence relation
\[(x,y) \sim (x',y') \quad \iff \quad y' = [x',x,y].\] 
Since $H\times H$ is a semigroup with product $(x,y)(x',y') \coloneqq (x, [y,x',y'])$ and $\sim$ is a congruence on $H \times H$, the quotient $H\times H / {\sim}$ is a semigroup. In fact, it is a group with identity the class of $(x,x)$. The inverse of the class of $(x,y)$ is the class of $(y,x)$. The assignment
\[\op{\tr(H)} \lra H \times H/{\sim}, \qquad \tau_x^y \lto \overline{(x,y)}\]
is an isomorphism of groups with inverse
\begin{equation*}
H \times H/{\sim} ~\lra \op{\tr(H)}, \qquad \overline{(x,y)} \lto \tau_x^y. \qedhere
\end{equation*}

\subsection{Trusses and their modules}\label{ssec:trusses}

A {\em truss} is an abelian heap $T$ together with an associative binary operation $\cdot: T\times T\to T,\ (s,t)\mapsto st,$ such that for all $s,s',s'',t\in T$,
\begin{enumerate}[label=(T\arabic*)]
\item $t[s,s',s'']=[ts,ts',ts'']$ \qquad (left distributivity),
\item $[s,s',s'']t=[st,s't,s''t]$ \qquad (right distributivity).
\end{enumerate}
A truss $T$ is said to be \emph{unital} if there exists a distinguished element $1_T \in T$ such that $t1_T = t= 1_Tt$ for all $t \in T$.

A {\em morphism of trusses} is a heap morphism $f \colon T\to T'$ between trusses which preserves the binary operation, i.e., for all $s,t\in T$, $f(st)=f(s)f(t)$. If $T$ and $T'$ are unital, then we often require $f$ to preserve the units as well, i.e., $f(1_T) = 1_{T'}$.
Trusses and their morphisms form a category that we denote by $\bf Trs$.

\begin{example}
    Let $R$ be a ring. The abelian heap $\hH(R)$ associated with the underlying abelian group structure, with the same multiplication, is a truss, that we denote by $\tT(R)$. If $R$ is unital, then $\tT(R)$ is unital as well. This construction induces a functor $\tT\colon \textbf{Rings} \to \textbf{Trs}$ from the category of rings to the category of trusses that also restricts to their unital counterparts.
\end{example}

A {\em left module over a truss } $T$ or a {\em left $T$-module} is an abelian heap $M$ together with an action $\cdot :T\times M\to M$ such that for all $t,t',t''\in T$ and $m,n,e\in M$,
\begin{enumerate}[label=(M\arabic*)]
\item $t\cdot (t'\cdot m)= (tt')\cdot m$,
\item $[t,t',t'']\cdot m=[t\cdot m,t'\cdot m,t''\cdot m]$,
\item $t\cdot[m,n,e]=[t\cdot m,t\cdot n,t\cdot e]$.
\end{enumerate}
An element $e$ in a left $T$-module $M$ is called an \emph{absorber} if $t \cdot e =e $ for all $t \in T$.
A module $M$ over a unital truss $T$ is said to be {\em unital} if $1_T \cdot m = m$ for all $m \in M$. The category of left $T$-modules and their morphisms will be denoted by $T\Mod$. 

Analogously, one defines ({\em unital}) {\em right $T$-modules}.  A {\em $(T,S)$-bimodule} is a left $T$-module and right $S$-module $M$ such that $(t\cdot m)\cdot s=t\cdot (m\cdot s)$ for all $t\in T$, $s\in S$ and $m\in M$.

As it happens for rings and modules, many facts about modules over trusses have a straightforward unital analogue. In what follows we will deal with not necessarily unital modules over not necessarily unital trusses, unless differently specified, leaving to the interested reader the task of specifying the results to the unital setting.

\begin{definition}\label{def:induced} 
Let $M$ be a non-empty left $T$-module. For every $e \in M$, the action $\triangleright_e\colon  T \times M \lra M$, given by
\[t\ \triangleright_e\  m \coloneqq [t\cdot m,t\cdot e,e], \qquad \textrm{for all }m\in M, t \in T,\]
is called the \emph{$e$-induced action} or the \emph{$e$-induced module structure} on $M$. We say that a subset $N\subseteq M$ is {\em an induced submodule} of $M$ if $N$ is a non-empty sub-heap of $M$ and $t\ \triangleright_e\  n \in N$ for all $t\in T$ and $n,e\in N$.
\end{definition}

\begin{theorem}[{\cite[Proposition 4.32]{Brz:par}}]
Let $M$ be a $T$-module and $N \subseteq M$ a non-empty sub-heap. The quotient $M/N$ is a $T$-module with the canonical map $\pi:M\to M/N$ being an epimorphism of $T$-modules if and only if $N$ is an induced submodule of $M$.
\end{theorem}

A distinguished class of $T$-modules is the one consisting of non-empty $T$-modules together with a fixed absorber. That is, abelian groups with a left $T$-action.

\begin{definition}[{\cite[Definition 2.30]{BrBrRySa}}]
Let $T$ be a truss. A {\it pointed module over $T$}, or {\it pointed $T$-module}, is an abelian group $G$ together with an action \mbox{$\cdot\colon T\times G\to G$} of the multiplicative semigroup of $T$ on $G$ such that for all $t,t',t''\in T$ and $g,h\in G$, 
\begin{equation}\label{eq:TGrp}
[t,t',t'']\cdot g = t\cdot g - t'\cdot g + t''\cdot g \qquad  \text{and} \qquad t\cdot (g+h) = t\cdot g+t\cdot h.
\end{equation}
If $T$ is unital and $1_T \cdot g = g$ for all $g \in G$, then $G$ is called a {\em unital} pointed $T$-module.

A morphism of pointed $T$-modules is by definition a group homomorphism $f \colon  G \to G'$ such that 
$f(t \cdot g) = t \cdot f(g)$
for all $g \in G$ and $t \in T$. For the sake of brevity, we may often call them \emph{$T$-linear group homomorphisms}.
All pointed $T$-modules together with $T$-linear group homomorphisms form the category $T\GMod$.
\end{definition}

\begin{remark}\label{rem:trussaction}
    Let $G$ be a pointed $T$-module and consider $\lambda\colon T \to \Set(G,G)$, $t \mapsto \lambda_t,$ where $\lambda_t(g) \coloneqq t \cdot g$ for all $g \in G$, $t \in T$. The right-hand side of \eqref{eq:TGrp} entails that $\lambda_t \in \Ab(G,G)$ for all $t \in T$. The left-hand side of \eqref{eq:TGrp} implies that $\lambda$ is a morphism of heaps, where the heap structure on $\Ab(G,G)$ is the one induced by its own abelian group structure (i.e.\ the point-wise one). The fact that $\cdot$ is a semigroup action means that $\lambda$ is also multiplicative and so that 
    \[\lambda \colon T \lra \tT\big(\Ab(G,G)\big), \qquad t \lto \lambda_t,\]
    is a morphism of trusses.
    Conversely, if $G$ is an abelian group, then every morphism of trusses  $\lambda \colon T \lra \tT\big(\Ab(G,G)\big)$, $t \lto \lambda_t,$ induces a pointed $T$-module structure on $G$, where $t\cdot g \coloneqq \lambda_t(g)$.
\end{remark}

\begin{example}
Consider the truss $Oint \coloneqq \big(\{2n+1\ |\ n\in \mathbb{Z}\},[-,-,-],\cdot\big)$ consisting of odd integers together with multiplication of integers and ternary bracket $[a,b,c] \coloneqq a-b+c$ for $a,b,c\in Oint$. Then, consider the group $(2\mathbb{Z},+)$ of even integers with usual addition. We have that $(2\mathbb{Z},+)$ is a pointed $Oint$-module with usual multiplication of integers as action, that is, $(2m+1)\cdot 2n \coloneqq 4mn+2n$, for all $m,n\in \mathbb{Z}$. 
\end{example}

\subsection{Trusses and rings}\label{ssec:TrandRing}
It is a well-known fact in truss theory (see \cite[\S5]{ABR} and \cite[Lemma 3.13]{BrzRyb:mod}) that there exists a ring $\rR(T)$, which is unital if $T$ is unital, satisfying the following universal property: there is a morphism of trusses $\iota_T \colon T \to \tT\rR(T)$ such that for any ring $R$ and any morphism of trusses $\varphi \colon T \to \tT(R)$ there exists a unique ring homomorphism $\hat{\varphi} \colon \rR(T) \to R$ making the following diagram commute:
\[
\xymatrix{
T \ar[rr]^{\iota_T} \ar[dr]_-{\varphi} && \tT\rR(T) \ar@{.>}[dl]^-{\exists!\tT(\hat{\varphi})} \\
& \tT(R) & 
}
\]
That is to say, the pair $(\rR(T), \iota_T)$ is a universal arrow in the sense of \cite[Chapter III, \S1, Definition]{Mac:lane} or, equivalently, the functor $\rR\colon \textbf{Trs} \to \textbf{Rings}$ is left adjoint to the functor $\tT\colon \textbf{Rings} \to \textbf{Trs}$ from \S\ref{ssec:trusses}. Moreover, if $T$ and $R$ are unital and $\varphi(1_T) = 1_R$, then $\hat{\varphi}$ preserves the units, too.

\begin{remark}[{\cite[\S5]{ABR}}]
\label{rem:rt}
    If $T$ is the empty truss, then $\rR(T) = 0$, the zero ring. Let $T$ be a non-empty truss and let $o \in T$. The ring $\rR(T)$ can be realised as the abelian group $\gG(T;o) \oplus \ZZ$ with multiplication
    \[(t,m)(s,n) = \big(ts + (n-1)to + (m-1)os + (m-1)(n-1)o^2,mn\big)\]
    and $\iota_T \colon T \to \tT\rR(T), t \mapsto (t,1)$ (here and elsewhere we simply write $+$ for the addition in $\gG(T;o)$).
    The truss $T$ is unital with unit $1_T$ if and only if the associated ring $\rR(T)$ is unital with unit $(1_T,1)$.
    
    If the truss $T$ is unital, we can choose $o=1_T$ and then the multiplication in the ring $\rR(T) = \gG(T;1_T) \oplus \ZZ$ takes the nicer form
    \begin{equation*}
    \begin{aligned}
    (t,m)(s,n) & = \big(ts + (n-1)t + (m-1)s + (m-1)(n-1)1_T,mn\big) \\
    & = \big(ts + (n-1)t + (m-1)s,mn\big)
    \end{aligned}
    \end{equation*}
    for all $t,s \in T$, $m,n \in \ZZ$. Observe that, in the latter situation, $\gG(T;1_T)$ is a ring without unit with respect to the multiplication
    \[t\cdot s \coloneqq ts - t - s\]
    for all $s,t \in T$ and $\rR(T)$ is its Dorroh extension.
\end{remark}

Readers accustomed to work with unital rings are familiar with the fact that with any not necessarily unital ring $R$, one may always associate a unital ring $R_u$ called its \emph{Dorroh extension}, that we just mentioned (see \cite{Dorroh}). This $R_u$ can be characterized as the universal unital ring associated with $R$ in the following sense: there is a canonical ring homomorphism $\jmath_R \colon R \to R_u$ such that for any unital ring $R'$ and any morphism  of rings $R \to R'$ there exists a unique morphism of unital rings $R_u \to R'$ making the obvious diagram commutative. Furthermore, the category of unital modules over $R_u$ is isomorphic (via the identity functor) to the category of $R$-modules.

An analogous construction exists for trusses (see \cite[\S2.5]{BrzRyb:fun}): for any not necessarily unital truss $T$, there exists a unital truss $T_u$ and a truss homomorphism $\jmath_T \colon T \to T_u$ which is universal in the sense that for any other unital truss $T'$ and any truss homomorphism $f \colon T \to T'$, there exists a unique morphism of unital trusses $\tilde f \colon T_u \to T'$ such that $\tilde f\circ \jmath_T = f$.

The next result shows that taking Dorroh extensions commutes with taking the ring associated to a truss.

\begin{proposition}\label{prop:Dorroh}
Let $T$ be a truss. Then $\rR(T)_u \cong \rR(T_u)$ as unital rings.
\end{proposition}

\begin{proof}
    This can be easily proved by taking advantage of the universal properties of these constructions. If we consider the composition
    \[
    T \xrightarrow{\jmath_T} T_u \xrightarrow{\iota_{T_u}} \tT\rR(T_u),
    \]
    then this is a morphism of trusses from $T$ to the truss associated with a ring. Hence, there exists a unique morphism of rings $\varphi \colon \rR(T) \to \rR(T_u)$ such that $\tT(\varphi) \circ \iota_T = \iota_{T_u} \circ \jmath_T$. Moreover, $\rR(T_u)$ being unital, there exists a unique morphism of unital rings $\hat \varphi \colon \rR(T)_u \to \rR(T_u)$ such that $\hat \varphi \circ \jmath_{\rR(T)} = \varphi$.

    In the opposite direction, the composition
    \[T \xrightarrow{\iota_T} \tT\rR(T) \xrightarrow{\tT(\jmath_{\rR(T)})} \tT(\rR(T)_u)\]
    is a morphism of trusses landing in a unital truss, whence there exists a unique morphism of unital trusses $\psi \colon T_u \to \tT(\rR(T)_u)$ such that $\psi \circ \jmath_T = \tT(\jmath_{\rR(T)}) \circ \iota_T$. Since, in turn, $\psi$ lands into the truss associated with a unital ring, there exists a unique morphism of unital rings $\hat \psi \colon \rR(T_u) \to \rR(T)_u$ such that $\tT(\hat \psi) \circ \iota_{T_u} = \psi$.

    By the uniqueness part of the respective universal properties, $\hat \varphi$ and $\hat \psi$ are each other inverses.
\end{proof}

\subsection{Heaps of modules}\label{ssec:HoM}

Let $T$ be a truss. A {\em heap of $T$-modules} $(M,[-,-,-],\triangleright)$ is an abelian heap $(M,[-,-,-])$ together with an operation 
\[\triangleright : T \times M \times M \to M, \qquad (t,m,n) \mapsto t\triangleright_m n,\]
such that
\begin{enumerate}[label=(HM\arabic*),ref=(HM\arabic*),leftmargin=1.5cm]
\item\label{item:HMM}
for all $m\in M$, the operation 
\[
\triangleright_m:T\times M \to M, \qquad (t,n)\mapsto t\triangleright_m n,
\]
makes $M$ a left $T$-module,
\item\label{item:HM3} the operation $\triangleright$ satisfies the {\em base change property}
\begin{equation}\label{eq:basechange}
t\triangleright_m n= [t\triangleright_e n, t\triangleright_e m,m],
\end{equation}
for all $m,n,e \in M$, $t \in T$.
\end{enumerate}
For the sake of brevity, we will often denote a heap of $T$-modules simply by $(M,\triangleright)$, leaving the heap structure $[-,-,-]$ understood.

A {\em morphism of heaps of modules} is a morphism $f : M \to N$ of heaps such that
\begin{equation}\label{eq:morph}
f(t\triangleright_m n) = t\triangleright_{f(m)}f(n),
\end{equation}
for all $m,n \in M$, $t \in T$. The category of heaps of $T$-modules and their morphisms will be denoted by $T\HMod$.

Suppose that $T$ is unital with unit $1_T$. A heap of $T$-modules $M$ is said to be \emph{isotropic} if $\tra {1_T} m n = n$ for all $m,n\in M$ (see \cite[\S3.2]{BrBrRySa}). The full subcategory of isotropic heaps of $T$-modules is denoted by $T\HMod_{\bf is}$.

\begin{remark}[{\cite[\S3.1]{BrBrRySa}}]
Among the consequences of the base change property \eqref{eq:basechange} one finds that
$\tra tmm=m$
and
$\tra tmn = [n,\tra tnm,m]$ for all $m,n\in M$, $t \in T$.
Furthermore, for all $e,f\in M$, the translation automorphism $\tau_e^f:M\to M$ 
is a morphism of heaps of modules, that is, it satisfies \eqref{eq:morph}.
\end{remark}

The following remark tells us that the notation for the heap of modules structure and that from \cref{def:induced} are consistent.

\begin{remark}\label{rem:M-HoM}
Let $(M,\cdot)$ be a $T$-module. Then, the operation $\triangleright\colon T \times M \times M \to M$ given by $t \triangleright_mn \coloneqq [t \cdot n, t \cdot m ,m ]$ for all $t \in T$, $m,n \in M$ makes $M$ into a heap of $T$-modules.
Moreover, if $(M,\triangleright)$ is a heap of $T$-modules, then the base change property \eqref{eq:basechange} means exactly that for all $e,m\in M$, the module structure $\triangleright_m$ is the $m$-induced module structure on $(M,\triangleright_e)$. 
Hence a heap of $T$-modules can be understood as a family of $T$-modules $\{(M,\triangleright_e) \mid e\in M\}$ in which each module is an induced module of any member of the family.
\end{remark}

The next example is worth mentioning as a motivation to study heaps of modules.

\begin{example}
\label{ex:ODE}
Every affine space over a field $\k$ is a heap of $\tT(\k)$-modules. In particular, the space of solutions of a linear differential equation of the form
\[
P(y) \coloneqq a_n(x)y^{(n)} + \cdots + a_0(x)y + r(x) = 0,
\]
is a heap of $\tT(\k)$-modules with respect to the structures
\[
\begin{gathered}
[f,g,h] \coloneqq f - g + h \qquad \text{and} \qquad \tra kgf \coloneqq kf - kg + g,
\end{gathered}
\]
for all $k \in \k$ and all $f,g,h$ solutions. 
\end{example}

\begin{example}\label{ex:ahTHmod}
    The heaps of modules over $T = \varnothing$ are exactly the abelian heaps, since there is one and only one morphism $\varnothing \times H \times H = \varnothing \to H$ for every abelian heap $H$. This induces, in fact, an isomorphism of categories $\Ah \cong \varnothing\HMod$.
\end{example}

Triggered by \cref{rem:M-HoM}, the following is a variation of \cite[Proposition 3.9 and Lemma 3.12]{BrBrRySa} which makes more explicit the isomorphism between the category of pointed heaps of $T$-modules (i.e. heaps of modules with a chosen point and morphisms preserving those points) and the category of pointed $T$-modules \cite[Proposition 4.5]{BrBrRySa}.

\begin{remark}\label{rem:TgrpsTHmods}
Let $T$ be a truss and let $G$ be a pointed $T$-module. If we consider the map 
\[\triangleright \colon T\times G\times G \lra G, \qquad (t,x,y) \lto t\triangleright_x y = t \cdot y - t \cdot x + x,\]
then $\Hh(G) \coloneqq (\hH(G),\triangleright)$ is a non-empty heap of $T$-modules. Furthermore, the assignment $\Hh \colon (G,\cdot) \lto \big(\hH(G),\triangleright\big)$ from the category of pointed $T$-modules to the category of non-empty heaps of $T$-modules is functorial. In the opposite direction, let $(M,\triangleright)$ be a non-empty heap of $T$-modules. Then for any chosen $e\in M$ we can consider 
\[\triangleright_e:T\times M\to M\] 
and the pair $\big(\gG(M;e),\triangleright_e\big)$ turns out to be a pointed $T$-module, which we denote by $\Gg\big((M,\triangleright);e\big)$. The assignment $\Gg \colon (M,\triangleright ) \lto \big(\gG(M;e),\triangleright_e\big)$ from the category of non-empty heaps of $T$-modules to the one of pointed $T$-modules is itself functorial, if for every morphism of heaps of $T$-modules $\varphi \colon (M,\triangleright) \to (N,\triangleright)$ and for every chosen $e \in M$ and $f \in N$ we consider
\[\Gg(\varphi) \coloneqq \tau_{\varphi(e)}^f\circ \varphi \colon \big(\gG(M;e),\triangleright_e\big) \lra \big(\gG(N;f),\triangleright_f\big).\]
Moreover, one may observe that $\Hh\Gg\big((M,\triangleright);e\big) = (M,\triangleright)$ for every non-empty heap of $T$-modules $(M,\triangleright)$ and for all $e \in M$, and $\Gg\big(\Hh(G,\cdot);e\big) = \big(\gG(\hH(G);e),\triangleright_e\big) \cong (G,\cdot)$ via the translation automorphism $\tau_e^{0_G}$, for every pointed $T$-module $(G,\cdot)$.
\end{remark}

We conclude the section with an equivalent description of heaps of modules which will play a key role at the end of \cref{HOM:CA}.

Recall that if $H,H'$ are abelian heaps, then we may perform the tensor product $H \otimes H'$ of abelian heaps and this satisfies the analogue properties of the tensor product of modules. Since the base change property \eqref{eq:basechange} entails that, for fixed $t\in T$ and $n\in M$, the map $M\to M$, $m\mapsto \tra tmn$ is a heap homomorphism as well, a heap of $T$-modules can be described as an abelian heap $(M,[-,-,-])$ together with a heap homomorphism
\begin{equation}\label{eq:triangle}
\triangleright : T \otimes M \otimes M \to M, \qquad t \otimes m \otimes n \mapsto t \triangleright_mn
\end{equation}
which is $T$-associative
\begin{equation}\label{eq:Tassoc}
ts \triangleright_m n = t \triangleright_m \Big(s \triangleright_m n \Big)
\end{equation}
and satisfies the base change property
\begin{equation}\label{eq:baseChange}
t \triangleright_m n = \big[t\triangleright_en,t \triangleright_em,m\big]
\end{equation}
for all $m,n,e \in M$ and $t,s \in T$. 
If now we take advantage of the fact that $- \otimes M \colon \Ah \to \Ah$ is left adjoint to $\Ah(M,-)$ for every abelian heap $M$ (see \cite[Proposition 3.7]{BrzRyb:fun}) and that $\Ah$ is a symmetric monoidal category (see \cite[Remark 3.11]{BrzRyb:fun}), then the datum of \eqref{eq:triangle} is equivalent to the datum of a heap homomorphism
\begin{equation}\label{eq:Delta}
\Delta\colon M \lra \Ah(T,\eE(M))
\end{equation}
where $\eE(M) = \Ah(M,M)$ is the truss of endomorphisms of the abelian heap $M$. Essentially, $\Delta(m)(t) = t\triangleright_m-$ for all $m \in M$, $t \in T$. In this setting, the $T$-associativity \eqref{eq:Tassoc} is equivalent to the fact that $\Delta$ takes values in $\trs(T,\eE(M))$, i.e. $\Delta(m)$ is a morphism of trusses for every $m \in M$.

\begin{lemma}
    Given an abelian heap $M$ together with a $\triangleright$ structure as in \eqref{eq:triangle}, we have that it satisfies the base change property \eqref{eq:baseChange} if and only if $\Delta$ from \eqref{eq:Delta} satisfies
    \begin{enumerate}[label=(\alph*),leftmargin=0.8cm]
        \item\label{item:2.18a} $\Delta(e)(t)\circ c_e = c_e$, the constant map $c_e \colon M \to M, m \mapsto e$, for all $e \in M$, $t \in T$, and
        \item\label{item:2.18b} $\tau_e^f \circ \Delta(e)(t) \circ \tau_f^e = \Delta(f)(t)$ for all $e,f \in M$, $t \in T$.
    \end{enumerate}
\end{lemma}

\begin{proof}
    Let us begin by supposing that $\triangleright$ satisfies the base change property. Then
    \[\big(\Delta(e)(t)\circ c_e\big)(m) = t \triangleright_e e = e = c_e(m)\]
    and 
    \[
    \begin{aligned}
    \big(\tau_e^f \circ \Delta(e)(t) \circ \tau_f^e\big)(m) & = \big[t \triangleright_e[m,f,e], e,f\big] = \big[t \triangleright_em,t \triangleright_ef,e, e,f\big] \\
    & = t \triangleright_fm = \Delta(f)(t)(m)
    \end{aligned}
    \]
    for all $m,e,f\in M$, $t \in T$. Conversely, suppose that $\Delta$ satisfies \cref{item:2.18a} and \cref{item:2.18b}. Then
    \[
    \begin{aligned}
    \big[t\triangleright_en,t \triangleright_em,m\big] & = \big[\Delta(e)(t)(n),\Delta(e)(t)(m),m\big] \stackrel{\ref{item:2.18a}}{=} \big[\Delta(e)(t)\big(\tau_m^e(n)\big),e,m\big] \\
    & \stackrel{\ref{item:2.18b}}{=} 
    \Delta(m)(t)(n) = t \triangleright_m n
    \end{aligned}
    \]
    for all $e,m,n \in M$, $t \in T$.
\end{proof}

Summing up,

\begin{proposition}\label{prop:HtM}
    An abelian heap $M$ is a heap of $T$-modules if and only if there exists a morphism of abelian heaps 
    \[\Delta \colon M \to \trs\big(T,\eE(M)\big)\]
    such that
    \begin{enumerate}[label=(\alph*),leftmargin=0.8cm]
        \item\label{item:2.19a} $\Delta(e)(t)\circ c_e = c_e$, the constant map $c_e \colon M \to M, m \mapsto e$, for all $e \in M$, $t \in T$, and
        \item\label{item:2.19b} $\tau_e^f \circ \Delta(e)(t) \circ \tau_f^e = \Delta(f)(t)$ for all $e,f \in M$, $t \in T$.
    \end{enumerate}
\end{proposition}

\section{Heaps of modules over a truss and affine modules over a ring}\label{Tam}

Aiming at studying the categorical aspects of the theory of pointed $T$-modules and heaps of $T$-modules, let us relate these with modules and affine modules over rings, respectively.

\subsection{Pointed modules over a truss and modules over a ring}\label{ssec:TgrpsRTmods}

Let $T$ be a truss and recall from \S\ref{ssec:TrandRing} the existence of its associated universal ring $\rR(T)$.

\begin{theorem}\label{thm:dummyP}
The category $T\GMod$ of pointed $T$-modules is isomorphic to the usual, ring-theoretic, category $\rR(T)\Mod$ of $\rR(T)$-modules. In particular, it is an abelian category. Furthermore, if $T$ is unital, then the foregoing isomorphism restricts to the categories of unital pointed $T$-modules and unital $\rR(T)$-modules as well.
\end{theorem}

\begin{proof}
    A pointed $T$-module $G$ is an abelian group with a truss homomorphism $T \to \tT\big(\Ab(G,G)\big)$ (see Remark \ref{rem:trussaction}). Since $\Ab(G,G)$ is a ring, this induces a unique ring homomorphism $\rR(T) \to \Ab(G,G)$ as above, making of $G$ an $\rR(T)$-module. Conversely, given an $\rR(T)$-module $M$ with module structure given by $\mu\colon \rR(T) \to \Ab(M,M)$, the composition 
    \[T \xrightarrow{\iota_T} \tT\rR(T) \xrightarrow{\tT(\mu)} \tT(\Ab(M,M)) \]
    makes $M$ into a pointed $T$-module. These two constructions are functorial and one the inverse of the other.
\end{proof}

Therefore, in the category of pointed $T$-modules all categorical (co)limits can be constructed as in a category of modules over a ring, that is, by realising them in the category of abelian groups and then endowing them with the natural $T$-action.

\begin{remark}\label{rem:dummyP}
If a truss $T$ is non-empty and $\rR(T)$ is constructed on $\gG(T;o) \oplus \ZZ$ as in Remark~\ref{rem:rt}, then the action of $\rR(T)$ on a pointed $T$-module $G$ comes out as
$$
(t,n)\cdot g = t\cdot g + (n-1)o\cdot g,
$$
for all $t\in T$, $n\in \ZZ$ and $g\in G$.   
Conversely, if $M$ is a $\rR(T)$-module, then it is a pointed $T$-module with the action
$$
t\cdot m = (t,1)\cdot m,
$$
for all $t\in T$ and $m\in M$.
\end{remark}

\begin{remark}\label{rem:sub-pointed-generated}
Let $G$ be a pointed $T$-module. The pointed $T$-submodule $\langle X \rangle_T$ generated by a set $X \subseteq G$ is the $\rR(T)$-submodule of $G$ generated by $X$. For instance, if $X = \{x\}$, then
\[
\begin{aligned}
\langle x \rangle_T & = \big\{r \cdot x + mx \mid r \in \rR(T) ,m \in \ZZ\big\} \\
& = \big\{a_1(t_1\cdot x) + \cdots + a_n(t_n \cdot x) + mx \mid n\in\NN, t_1,\ldots,t_n \in T, a_1,\ldots,a_n,m \in \ZZ\big\}.
\end{aligned}
\]

If we fix, without loss of generality, $e \in T$, then
\[\langle x \rangle_T = \big\{t\cdot x + n(e\cdot x) + mx \mid m,n\in\ZZ, t \in T\big\},\]
since $t\cdot x + s\cdot x = [t,e,s]\cdot x + e\cdot x$ for all $s,t \in T$.
\end{remark}

\begin{proposition}\label{prop:freeTgroup}
    Let $T$ be a truss. The forgetful functor $U \colon T\GMod \to \Set$ admits a left adjoint $F_T \colon \Set \to T\GMod$ given, on objects, by sending any set $X$ to the pointed $T$-module ${\rR(T)_u}^{(X)} = \bigoplus_{x \in X}\rR(T)_u$, where $\rR(T)_u$ is the Dorroh extension of $\rR(T)$. 
\end{proposition}

\begin{proof}
    Notice that the isomorphism between $T\GMod$ and $\rR(T)\Mod$ from \cref{thm:dummyP} is compatible with the forgetful functors to $\Set$. Therefore, the left adjoint to $U$ is uniquely determined by sending any $X$ to the $\rR(T)_u$-module $\rR(T)_u \otimes_\ZZ \ZZ^{(X)} \cong {\rR(T)_u}^{(X)}$ and then restricting the action along the composition $T \to \rR(T) \to \rR(T)_u$.    
\end{proof}

\begin{remark}
For unital pointed modules over a unital truss $T$, the free unital pointed $T$-module functor is given, on objects, by sending any set $X$ to the pointed $T$-module $\rR(T)^{(X)}$.
\end{remark}

\begin{corollary}\label{cor:freeTgrp*}
{\rm (a)} If $T = \varnothing$, then the free pointed $T$-module $F_T(*)$ over the singleton $\{*\}$ is the free abelian group $\ZZ$. 

{\rm (b)} If $T$ is not empty and $o \in T$, then $F_T(*)$ is the abelian group $\gG(T;o)\oplus\ZZ \oplus \ZZ$ with left $T$-action
    \[t \cdot (s,n,p) = (ts + (n-1)to + pt, n + p , 0)\] 
    for all $s,t \in T$, $n,p \in \ZZ$. 
    
{\rm (c)} If $T$ is unital, then the free object over the singleton set in the category of unital pointed $T$-modules is the abelian group $\gG(T;1_T)\oplus\ZZ$ with left $T$-action
    \[t \cdot (s,n) = (ts + (n-1)t, n)\]
    for all $s,t \in T$, $n \in \ZZ$. 
\end{corollary}

\begin{proof}
For the reader's convenience, let us exhibit an explicit proof of the second claim. Let $T$ be a truss and let us fix $o\in T$, then $\rR(T)\cong \gG(T;o)\times \mathbb{Z}$ is a $T$-module with the following action
\begin{equation}\label{eq:RTgen}
     t\cdot (s,n)\coloneqq(t,1)(s,n)=(ts+(n-1)to,n) \textrm{ for all } {t,s\in T} \textrm{ and } {n\in \mathbb{Z}};
\end{equation}
see \cref{rem:rt}. Moreover, since $\rR(T)$ is a ring its Dorroh extension $\rR(T)_u$ is a module over $\rR(T)\cong \gG(T;o)\times \mathbb{Z}\times \mathbb{Z}$. Thus, by the restriction of the operations from \cref{rem:rt} and the usual Dorroh extension, we obtain a $T$-module action
$$
\begin{aligned}
 t\cdot(s,n,p)&\coloneqq ((t,1),0)\cdot ((s,n),p)=((t,1)(s,n)+(t,1)p,0)\\ &=(ts+(n-1)to+pt,n+p,0)),
\end{aligned}
$$
for any $t,s\in T$ and $n,p\in \ZZ$. 
This is a free object with a basis $\{(o,0,1)\}$. To see this, it is enough to observe that every element $(t,n,p)\in R(T)_u$ has a unique decomposition of the form $$
(t,n,p)= \alpha\cdot(o,0,1)+u(o\cdot (o,0,1))+v(o,0,1),
$$ 
with $\alpha\in T$, and $u,v\in \mathbb{Z}$. More precisely, $\alpha=t$, $u=n-1$ and $v=p$. Therefore, for every pointed $T$-module $M$ and every $g\in M$, there exits at most one morphism of pointed $T$-modules $\phi :R(T)_u\to M$ such that $\phi (o,0,1)=g$. This morphism is defined by the formula 

\begin{equation}\label{uni:form}
  \phi (t,n,p):=t\cdot g+(n-1)(o\cdot g)+pg, \quad  \textrm{ where } g=\phi (o,0,1).  
\end{equation}

Let us show that it is a morphism of pointed $T$-modules. First of all
    \[
    \begin{aligned}
    \phi\big((t,m,u) + (s,n,v)\big) & = 
    \phi(t+s,m+n,u+v)\\ &=
    [t,o,s]\cdot g+ (m+n-1)(o\cdot g)+(u+v)g\\ &=
    t\cdot g -o\cdot g+s\cdot g+m(o\cdot g)+(n-1)(o\cdot g)+ug+vg\\ &=\phi(t,m,u)+\phi (s,n,v),
    \end{aligned}
    \]
    whence it is a morphism of abelian groups. Furthermore, since in $M$
    \[
    \begin{aligned}
    &(nt)\cdot g = [t,o,\ldots,t,o,t]\cdot g = n(t\cdot g) - (n-1)(o \cdot g)\,, & \mbox{for}\; n \geq 0, \\
    &(nt)\cdot g = [o,t,o,\ldots,t,o]\cdot g = (-n+1)(o\cdot g) +n (t\cdot g)\,, & \mbox{for}\; n < 0,
    \end{aligned}
    \]
     for all $n,p \in \ZZ$, it follows that
    \[
    \begin{aligned}
        \phi\big(t\cdot (s,n,p)\big) & = \phi\big(ts + (n-1)to+pt,n+p,0\big) \\ & =  \big(ts + (n-1)to+pt\big)\cdot g + (n+p-1)(o\cdot g) \\
        &= ts\cdot g-o\cdot g+((n-1)to)\cdot g-o\cdot g+(pt)\cdot g+(n+p-1)(o\cdot g)\\
        &= ts\cdot g+(n-1)(to\cdot g)-(n-2)(o\cdot g)-2(o\cdot g)+p(t\cdot g)\\
        &-(p-1)(o\cdot g)+ n(o\cdot g)+(p-1)(o\cdot g)\\
        & = ts\cdot g +(n-1)(to \cdot g)+p(t\cdot g) \\
        & = t\cdot (s\cdot g + (n-1)(o \cdot g)+p\cdot g) = t \cdot \phi(s,n,p).
    \end{aligned}
    \]
    Therefore, $\phi$ is a morphism of pointed $T$-modules. The uniqueness of $\phi$ follows directly by the fact that $\phi$ is uniquely given by the image of $(o,0,1)$ in the formula \eqref{uni:form}.
\end{proof}

\begin{example}
Let us consider the truss $\mathbb{Z}_{6,3}\coloneqq(6\ZZ+3,[-,-,-],\cdot)$, where $[a,b,c]=a-b+c$, for all $a,b,c\in 6\ZZ+3$, and $+,\cdot$ are the addition and multiplication of integers, respectively. The free $\ZZ_{6,3}$-module over $\{*\}$ is the group $\gG(\ZZ_{6,3};3)\times \ZZ\times \ZZ$ together with the $\ZZ_{6,3}$-module operation, expressed in terms of the standard addition and multiplication of integers as
$$
(6z+3)\cdot (6t+3,n,p):=(36zt+18t+18zn+6n+6pz+3,n+p,0),
$$
where $z,t,n,p\in \ZZ$.
\end{example}

\subsection{Heaps of modules over a truss and affine modules over a ring}

Let $R$ be a not necessarily unital ring. Then, by adapting \cite[Definition 4.8]{BrBrRySa} to the present setting, we refer to a heap of $\tT(R)$-modules $(M,\triangleright)$ such that
$\tra {0_R}mn = m$ for all $m,n \in M$   as to an \emph{affine module} over $R$. If $R$ is unital, then we also require that $M$ is isotropic, that is, $\tra {1_R} m n = n$ for all $m,n \in M$ (in the unital setting, this definition is equivalent to that in \cite[\S1]{Ostermann-Schmidt}, see \cite[Proposition 4.9]{BrBrRySa}). A morphism of affine $R$-modules is simply a morphism of the corresponding heaps of $\tT(R)$-modules.

\begin{theorem}\label{thm:TheapRTaff}
    The category of heaps of $T$-modules is isomorphic to the category of affine $\rR(T)$-modules.
\end{theorem}

\begin{proof}
    Since the empty heap of $T$-modules corresponds to the empty affine $\rR(T)$-module, we may work with non-empty objects without loss of generality. A (non-empty) heap of $T$-modules $(M, \triangleright)$ is converted into an affine $\rR(T)$-module $(M,\tilde{\triangleright})$ by setting, for all $m,n\in M$ and $t\in T$,
    \begin{equation}\label{eq.affine.rt}
        t ~\tilde{\triangleright}_m~ n = \tra tmn \quad \& \quad 0~\tilde{\triangleright}_m~ n = m,
    \end{equation}
    and extending uniquely to the whole of $\rR(T)$. To see that this action is well-defined, one can first use a specific realisation of $\rR(T)$ (for example the one in Remark~\ref{rem:rt}, in which the zero of $\rR(T)$ is given by $(o,0)\in T\times \ZZ$ and $T$ is embedded as $(T,1)\subset T\times \ZZ$) and choose an element $e\in M$ to convert $(M, \triangleright)$ into a pointed $T$-module $(M,+_e,\triangleright_e)$. By \cref{thm:dummyP}, $(M,+_e,\triangleright_e)$ is an $\rR(T)$-module $(M,+_e,\cdot)$ with the action provided by Remark~\ref{rem:dummyP},
    $$
    (t,k)\cdot m = \tra tem +_e (k-1)\tra oem,
    $$
    for all $t\in T$, $k\in \ZZ$ and $m\in M$. To the latter module, we may assign a heap of $\tT\rR(T)$-modules by setting
    \[[m,n,p]_{+_e} \coloneqq m -_e n +_e p = [m,n,p]\]
    and $\tilde{\triangleright} \colon \tT\rR(T) \times M \times M \to M$  by
    \[
    \begin{aligned}
       (t,k) ~\tilde{\triangleright}_m~ n &\coloneqq (t,k) \cdot n -_e (t,k)\cdot m +_e m \\
       &= \big[\tra ten,\tra tem,m\big] +_e (k-1)\tra oe{(m-_en)}\\
       &= \tra tmn +_e (k-1)\tra oe{(m-_en)}, 
    \end{aligned}
    \]
    for all $t\in T$, $m,n\in M$ and $k\in \ZZ$. In particular, for elements $(t,1)$ and $(o,0)$ one obtains formulae \eqref{eq.affine.rt}, from which the above form of $~\tilde{\triangleright}$ has been uniquely derived.
 
 The formulae \eqref{eq.affine.rt} indicate that the assignment sending any heap of $T$-modules $(M,\triangleright)$ to the affine $\rR(T)$-module $(M,\tilde{\triangleright})$  induces the functor
    \[T\HMod \lra \textbf{Aff}_{\rR(T)},\]
    which acts on morphisms as the identity.
    
    In the opposite direction, any heap of $\rR(T)$-modules $(N,\triangleright)$ becomes a heap of $T$-modules $(N,\hat{\triangleright})$ by restriction of scalars along $\iota_T \colon T \to \tT\rR(T)$:
    \[t\ \hat{\triangleright}_m n \coloneqq \tra tmn\, ,\]
    for all $m,n\in M$ and for all $t \in T$. This construction is functorial, too, inducing 
    \[\textbf{Aff}_{\rR(T)} \lra T\HMod.\]
    It is clear that if we extend the scalars from $T$ to $\rR(T)$ and then we restrict them again to $T$, we find the same heap of $T$-modules we started with and that this equality is natural. Analogously, any affine $\rR(T)$-module $(M,\triangleright)$ has $\triangleright$ uniquely determined by the elements $\tra tmn$, because $\tra0mn = m$ for all $m,n \in M$, and therefore the proof is concluded. 
\end{proof}

\section{Categorical aspects of heaps of modules}\label{HOM:CA}

In this section we explore the categorical aspects of heaps of modules over a truss.


\subsection{Limits}

Heaps of modules over a truss $T$, as well as (abelian) heaps and $T$-modules, form a category of $\langle \Omega,E \rangle$-algebras in the sense of universal algebra (see, e.g., \cite[page 120]{Mac:lane} for a short recall). As such, the category of heaps of modules over $T$ is monadic over $\Set$ (see, e.g., \cite[Chapter VI, \S8, Theorem 1]{Mac:lane}) and hence all limits in $T\HMod$ can be computed in $\Set$ and then endowed with the obvious heap of modules structure (see, e.g., \cite[Propositions 4.3.1]{Bor-Hand}). In what follows, we work explicitly with heaps of $T$-modules, but the interested reader may easily adapt the examples to the case of $T$-modules.

\begin{example}[Equalizers]
Let $f,g \colon M \to N$
 be a pair of morphisms of heaps of $T$-modules. Then $E \coloneqq \{x \in M\mid f(x)=g(x)\}$ is a sub-heap of $T$-modules of $M$. Since $E$ together with the inclusion map $i:E\to M$ is \textit{the equalizer} of the pair $(f,g)$ in the category of sets, it is easy to see that $(E,i)$ is also the equalizer of $(f,g)$ in the category $T\HMod$. We note that $E$ can be the empty heap. 
\end{example}

\begin{example}[Products]\label{ex:prods}
Let $(M_i)_{i\in I}$ be a family of heaps of $T$-modules. The product of the family $(M_i)_{i\in I}$ in $T\HMod$  can be realised by endowing their Cartesian product $\Prod_{i\in I}M_i$ with the component-wise heap of modules structure, together with the canonical projections $\pi_i:\Prod_{i\in I}M_i\to M_i$, $i\in I$. 

Notice that the heap bracket $[-,-,-]\colon M \times M \times M \to M$ becomes then a morphism of heaps of $T$-modules for every $M$ in $T\HMod$. Firstly, the fact that the heap $M$ is abelian entails that the bracket is a morphism of abelian heaps. Secondly, the base change property entails that, for all $e \in M$, 
\begin{equation}\label{eq:bracketTHeaps}
\begin{aligned}
t \triangleright_{[a,b,c]} [x,y,z] & \stackrel{\eqref{eq:basechange}}{=} \big[t \triangleright_e[x,y,z], t\triangleright_e[a,b,c],[a,b,c]\big] \\
& \stackrel{\phantom{\eqref{eq:basechange}}}{=} \big[t \triangleright_ex,t \triangleright_ey,t \triangleright_ez,t \triangleright_ea,t \triangleright_eb,t \triangleright_ec,a,b,c\big] \\
& \stackrel{\eqref{eq:basechange}}{=} \big[t \triangleright_a x, t\triangleright_by, t\triangleright_cz\big],
\end{aligned}
\end{equation}
which is exactly $[-,-,-]$ applied to $t\triangleright_{(a,b,c)}(x,y,z) = (t \triangleright_a x, t\triangleright_by, t\triangleright_cz)$.
\end{example}

\begin{example}[Pullbacks]
    Let $M \xrightarrow{f} O \xleftarrow{g} N$ be a pair of morphisms in $T\HMod$. Then $P \coloneqq \{(m,n) \in M \times N \mid f(m) = g(n)\}$ is a sub-heap of $T$-modules of the product $M \times N$ and, together with the obvious maps $\pi_M \colon P \to M, (m,n) \mapsto m,$ and $\pi_N \colon P \to N, (m,n) \mapsto n,$ it forms the pullback of the pair $(f,g)$.
\end{example}

\subsection{Colimits}

In this subsection we describe some (less obvious) colimits in the categories of (abelian) heaps, $T$-modules and heaps of $T$-modules. Existence of the colimits follows by the general theory of algebraic theories (see e.g. \cite[\S3.4]{Bor-Hand}), here we present their explicit descriptions.

\subsubsection{Modules}

Let $T$ be a truss. Colimits in the category of $T$-modules can be computed in the category of abelian heaps and then endowed with the obvious $T$-module structure. This is a consequence of the following result, which collects the observations and conclusions of \cite[pages 24 and 25]{BrzRyb:fun}, and of \cite[Propositions 4.3.2]{Bor-Hand}, in view of the fact that the tensor product $T \otimes - \colon \Ah \to \Ah$ is a left adjoint functor and hence preserves all colimits.

\begin{proposition}
Let $T$ be a truss.
\begin{enumerate}[leftmargin=0.8cm]
    \item\label{item:3.4.1} If $T$ is unital, then the forgetful functor from the category of unital $T$-modules to that of abelian heaps admits as left adjoint the free $T$-module functor $T \otimes-$ and it is monadic.
    \item\label{item:3.4.2} If $T$ is not necessarily unital, then the forgetful functor from the category of not necessarily unital $T$-modules to that of abelian heaps admits as left adjoint the free $T$-module functor $T_u \otimes- $, where $T_u$ is the unital extension of $T$, and it is monadic.
\end{enumerate}
\end{proposition}

\subsubsection{Heaps of modules}

Colimits in the category of heaps of $T$-modules are more delicate and cannot be reduced to colimits in $\Ah$, as it happens for $T$-modules. In fact, the converse is true: since $\Ah$ is (isomorphic to) $\varnothing\HMod$, as in Example \ref{ex:ahTHmod}, colimits of abelian heaps are particular cases of colimits of heaps of modules.

\begin{construction}[Quotients - {see \cite[Proposition 3.8]{BrBrRySa}}]\label{rem:quotients}
Let $(M,[-,-,-],\triangleright)$ be a heap of $T$-modules and let $N \subseteq M$ be a sub-heap of $T$-modules. Consider the sub-heap relation $\sim_N$ from \cref{rem:subheaprel}. For $t \in T$, $m_1 \sim_N m_1'$ and $m_2 \sim_N m_2'$, 
\[
\Big[t \triangleright_{m_1'} m_2', t \triangleright_{m_1} m_2,n\Big] = \Big[t \triangleright_{m_1'} m_2', t \triangleright_{m_1} m_2,t\triangleright_n n\Big] \stackrel{\eqref{eq:bracketTHeaps}}{=} t\triangleright_{[m_1',m_2',n]} [m_1,m_2,n] \in N ,
\]
for some $n \in N$ and hence $\sim_N$ is a congruence over $(M,[-,-,-],\triangleright)$. Therefore, the set of equivalence classes $M/N$ is a heap of $T$-modules with
\[\Big[\ol{m_1},\ol{m_2},\ol{m_3}\Big] \coloneqq \ol{[m_1,m_2,m_3]} \qquad \text{and} \qquad t\triangleright_{\ol{m_1}} \ol{m_2} \coloneqq \ol{t \triangleright_{m_1} m_2}\, ,\]
and the canonical projection $\pi\colon M \to M/N$ is a morphism of heaps of $T$-modules satisfying the natural universal property: if $g \colon M \to M'$ is a morphism of heaps of $T$-modules such that $g(n) = e'\in M'$ for all $n \in N$, then there exists a unique morphism of heaps of $T$-modules $\tilde{g} \colon M/N \to M'$ such that $\tilde{g} \circ \pi = g$.

Furthermore, if $\sim$ is any congruence on $M$, then $\sim~ = ~\sim_N$ for the sub-heap of modules
\[N \coloneqq \left\{m \in M \mid m \sim e \textrm{ for a fixed } e\in M\right\} \subseteq M.\]
Indeed, we already know that $N$ has to be a (normal) sub-heap and, in addition, for $t \in T$ and $n,n'\in N$,
\[t \triangleright_n n' \sim t \triangleright_e e = e \in N,\]
so that $t \triangleright_n n' \in N$.
\end{construction}

In the following we will rely on the correspondence between heaps of $T$-modules and pointed $T$-modules as it is described in \cite{BrBrRySa} (see Remark \ref{rem:TgrpsTHmods}). In order to simplify the notation, if there is no danger of confusion, we suppress the index $e$ and we denote by $(M,+)$ the pointed $T$-module associated with $(M,[-,-,-],\triangleright)$. Recall also that with these notations, every morphism $f:M\to N$ of heaps of $T$-modules has the form $f=\alpha+f(0)$, where $\alpha:(M,+)\to (N,+)$ is a morphism of pointed $T$-modules.

\begin{proposition}[Coequalizers]\label{prop:coeq}
For every pair of parallel arrows $(f,g)$ from $G$ to $H$ in $T\HMod$, let $[[ f,g]]_e$ denote the sub-heap of $T$-modules of $H$ generated by a chosen $e \in H$ and the elements $[f(x),g(x),e]$ for all $x \in G$. Then, the coequalizer of $(f,g)$ can be realized as $H/[[f,g]]_e$ with the canonical projection.
\end{proposition}

\begin{proof}
    For the sake of simplicity, we will take advantage of the conventions from the end of \cref{ssec:HoM}. Let $f,g \colon G \to H$ be two parallel morphisms of heaps of $T$-modules. If $G=\varnothing$ then it is easy to see that the coequalizer of the pair $(f,g)$ is $H$    with the identity map. Suppose then that $G$ is not empty, whence $H\neq \varnothing$, too. Without loss of generality, pick $e \in H$ and consider the set
    \[[[ f,g ]] \coloneqq \big\{[f(x),g(x),e] \in H \mid x \in G\big\}.\]
    This is a sub-heap of $T$-modules of $H$ since
    \[
    t\triangleright_{[f(x),g(x),e]}\big[f(y),g(y),e\big] \stackrel{\eqref{eq:bracketTHeaps}}{=} \big[t \triangleright_{f(x)} f(y), t\triangleright_{g(x)} g(y), t \triangleright_e e\big] = \big[f(t \triangleright_xy),g(t\triangleright_x y),e \big] ,
    \]
    for all $x,y \in G$. Let $[[ f,g]]_e$ be the sub-heap of $T$-modules of $H$ generated by $[[ f,g]]$ and $e$. Then we claim that $\mathsf{coeq}(f,g) = H/[[ f,g]]_e$ with the canonical projection $\pi \colon H \to H/[[ f,g]]_e$ as coequalizer morphism. First of all, since $e \in [[ f,g]]_e$ and $[f(x),g(x),e] \in [[ f,g]]_e$, too, \cref{rem:quotients} entails that $\pi(f(x)) = \pi(g(x))$ for all $x \in G$. Secondly, if $h \colon H \to K$ is a morphism of heaps of $T$-modules such that $h \circ f = h \circ g$, then
    \[h\big([f(x),g(x),e]\big) = \big[hf(x),hg(x),h(e)\big] = h(e).\]
    This implies that $h\big([[ f,g]]\big) = \{h(e)\}$, that $h\big(t \triangleright_u v\big) = t \triangleright_{h(u)} h(v) = t \triangleright_{h(e)} h(e) = h(e)$ for all $u,v \in [[ f,g]] \cup\{e\}$ and hence that $h\big([[ f,g]]_e\big) = \{h(e)\}$. Thus, by \cref{rem:quotients}, there exists a unique morphism $\tilde{h} \colon H/[[ f,g]]_e \to K$ such that $\tilde{h} \circ \pi = h$.
\end{proof}

\begin{remark}
Let $M$ be a non-empty heap of $T$-modules, $e \in M$, and $N \subseteq M$ be a sub-heap of $T$-modules. For the sake of making the construction from the proof of \cref{prop:coeq} a bit more concrete, let us explicitly construct the sub-heap of $T$-modules $N_e$ generated by $N$ and $e$. Consider the sub-heap $P \subseteq M$ generated by the set $X \coloneqq \{e\} \cup N \cup \{t \triangleright_e n \mid n\in N, t \in T\}$, that is, the set $W(X)$ of all the odd-length brackets $[x_1, x_2, \cdots, x_{2n+1}]$ in $M$ with $x_1,\ldots,x_{2n+1} \in X$. Since for all $s \in T$ we have that
\[s \triangleright_e e = e \in X, \qquad s \triangleright_e n \in X, \qquad s \triangleright_e t \triangleright_e n = st \triangleright_e n \in X, \qquad \big(\forall\,n \in N, t\in T\big)\]
and since $s \triangleright_e (-)\colon M \to M$ is a morphism of heaps, $W(X)$ is closed under the ternary operation 
\[
T \otimes M \otimes M \lra M, \qquad t \otimes a \otimes b \lto t \triangleright_a b = [t \triangleright_e b, t \triangleright_e a ,a],
\]
and so it is a sub-heap of $T$-modules. It is clearly the smallest one containing $e$ and $N$, whence $W(X) = N_e$. A less canonical, but equivalent, construction in case $N \neq \varnothing$ allows us to realise $N_e$ as the subheap generated by the set $X' \coloneqq N \cup \{e\} \cup \{t \triangleright_f e \mid t \in T\}$ for $f \in N$ (the construction does not depend on the choice of $f$). That is to say, for $N \neq \varnothing$ and $f \in N$, $N_e$ can be equivalently described as the $\rR(T)$-submodule of $(M,+_e,\cdot_e)$ generated by the congruence class $f + \tau_f^e(N)$ of the equivalence modulo the $\rR(T)$-submodule $\tau_f^e(N)$, or as the $\rR(T)$-submodule of $(M,+_f,\triangleright_f)$ obtained by appending $e$ to the submodule $N$.
\end{remark}

\begin{construction}[Coequalizers]\label{con:coeq}
Let us explicitly construct the coequalizer for the pair of morphisms of heaps of $T$-modules
 $f,g \colon G \to H$ by taking advantage of the associated pointed $T$-module construction. 
Suppose that $G$ is not empty. Hence $H\neq \varnothing$. In order to construct the coequalizer for the pair $(f,g)$, we fix two pointed $T$-module structures $(G,+)$ and $(H,+)$. Then $f=\alpha+a$ and $g=\beta+b$ with $\alpha,\beta\in T\GMod(G,H)$ and $a,b\in H$. 

Let $h:H\to K$ be a morphism of heaps of $T$-modules. We also fix a pointed $T$-module structure $(K,+)$. Then $h=\gamma+c$ with $\gamma\in T\GMod(H,K)$ and $c\in K$. We observe that 
$$hf=hg \quad \Leftrightarrow \quad \gamma\alpha+\gamma(a)+c=\gamma\beta+\gamma(b)+c \quad \Leftrightarrow \quad \gamma\alpha-\gamma\beta=\gamma(a)-\gamma(b).$$ 
Since $\gamma\alpha-\gamma\beta$ is a morphism of pointed $T$-modules, it follows that $hf=hg$ if and only if $\gamma\alpha=\gamma\beta$ and $\gamma(a)=\gamma(b)$. 

This suggests choosing as the group associated to the coequalizer of $(f,g)$ the factor group $C=H/( \Ima(\alpha-\beta)+ \langle a-b\rangle_T)$ (where $\langle a-b\rangle_T$ is the pointed $T$-submodule generated by $a-b$) together with the canonical surjection $\pi_C:H\to C$. 

Using the definition of $C$ it follows that $\pi_Cf=\pi_Cg$. Let $h=\gamma+c:H\to K$ be a morphism of heaps of $T$-modules such that $hf=hg$. Then  $\gamma\alpha=\gamma\beta$ and $\gamma(a)=\gamma(b)$. Then $\overline{h}:C\to K$, $\overline{h}(x+( \Ima(\alpha-\beta)+ \langle a-b\rangle_T))=\gamma(x)+c$ is well defined and $\overline{h}\pi_C=\gamma+c=h$. Moreover, since $\pi_C$ is a surjective map, it follows that $\overline{h}$ is the unique map $C\to K$ which provides a factorisation of $h$ through $\pi_C$. Hence $(C,\pi_C)$ is the coequalizer for the pair $(f,g)$, as the argument does not depend on the particular choice of the pointed $T$-module structures.
\end{construction}

Now, we would like to construct the coproduct of a family of heaps of $T$-modules. In order to do this, we need the following

\begin{lemma}\label{free-subgroup}
Let $C$ be a pointed $T$-module, and let $(c_j)_{j\in J}$ be a family of elements of $C$ such that for every pointed $T$-module $H$ and every family $(a_j)_{j\in J}$ of elements of $H$ there exists a morphism of pointed $T$-modules $\alpha:C\to H$ such that $\alpha(c_j)=a_j$. Then $C$ admits a decomposition $C=B\oplus F$, where $F$ is a free pointed $T$-module and $(c_j)_{j\in J}$ is a basis for $F$. 
\end{lemma}

\begin{proof}
Let $F=\langle c_j\mid j\in J\rangle_T\leq C$. Using the hypothesis, it follows that $F$ is free and that $(c_j)_{j\in J}$ is a basis for $F$. Moreover, there exists a morphism of pointed $T$-modules $\pi:C\to F$ such that $\pi(c_j)=c_j$ for all $j\in J$. If $\iota:F\to C$ is the inclusion map then $\pi\iota=\id_F$, hence $F$ is a direct summand of $C$ as an abelian group. Since $\pi$ and $\iota$ are morphisms of pointed $T$-modules, the conclusion is obvious.  
\end{proof}

\begin{construction}[Coproducts]\label{con:coprod}
The coproduct of an empty family is the initial object of the category of heaps of $T$-modules (the empty heap).  
Since the empty heap of $T$-modules is the initial object in our category, it is enough to study the coproduct of a family of non-empty members.

Let $(G_i)_{i\in I}$ be a family of non-empty heaps of $T$-modules. In order to construct a coproduct for it, we proceed as follows:
\begin{enumerate}[label=(\alph*),leftmargin=0.8cm]
\item for every $i\in I$ we fix a pointed $T$-module structure $(G_i,+)$ associated to $G_i$; 
\item we fix an index $i_0\in I$ and a family of symbols $(b_i)_{i\in I\setminus \{i_0\}}$;
\item let $C$ be the heap associated with the pointed $T$-module $(\oplus_{i\in I}G_i)\oplus F$, where $F$ is the free pointed $T$-module which has as a basis the family $(b_i)_{i\in I\setminus \{i_0\}}$;
\item if $u_i:G_i\to \oplus_{i\in I}G_i$ are the canonical morphisms, we define $\upsilon_{i_0}=u_{i_0}$, and $\upsilon_{i}=u_i+b_i$ for all $i\in I\setminus \{i_0\}$.  
\end{enumerate}
Then the pair $(C, (\upsilon_i)_{i\in I})$ represents a coproduct of the family $(G_i)_{i\in I}$ in $T\HMod$.

Indeed, let $\alpha_i:G_i\to H$, $i\in I$, be a family of morphisms of heaps of $T$-modules. We fix a pointed $T$-module structure on $H$. Then $\alpha_i=f_i+a_i$, where $\alpha_i$ are morphisms of pointed $T$-modules and $a_i\in H$, for all $i\in I$. There exists a unique morphism of pointed $T$-modules $f:\bigoplus_{i\in I}G_i\to H$ such that $f_i=fu_i$ for all $i\in I$. Moreover, we consider the morphism $\gamma:F\to H$ defined by the rules $\gamma(b_i)=a_i-a_{i_0}.$

We define $\alpha:C\to H$ by the rule 
$$\alpha(g+c)=f(g)+\gamma(c)+a_{i_0} \textrm{ for all } g\in \bigoplus_{i\in I}G_i \textrm{ and }c\in F.$$
Note that $\alpha$ is well defined since the representation of the elements of $C$ as $g+c$ with  $g\in \bigoplus_{i\in I}G_i$ and $c\in F$ is unique. Putting $b_{i_0}=0$, we have
$$\alpha \upsilon_i(x)=\alpha(u_i(x)+b_i)=fu_i(x)+\gamma(b_i)+a_{i_0}=f_i(x)+a_i,$$
hence $\alpha\upsilon_i=\alpha_i$ for all $i\in I$. 

Suppose that $\alpha':C\to H$ is a morphism such that $\alpha'\upsilon_i=\alpha_i$ for all $i\in I$. Keeping the pointed $T$-modules fixed in (a) for the $G_i$ and above for $H$, we observe that 
$$
\alpha'(g+b)=f'(g)+\gamma'(b)+a,
$$ 
with $f':\bigoplus_{i\in I}G_i\to H$ and $\gamma':F\to H$ morphisms of pointed $T$-modules, and $a\in H$.

From $\alpha'\upsilon_{i_0}=\alpha_{i_0}$ we obtain $f'u_{i_0}+a=\alpha_{i_0}$, hence $a=a_{i_0}$ and $f'u_{i_0}=f_{i_0}$. 

For $i\neq i_0$ we obtain $f'u_i+\gamma(b_i)+a_{i_0}=f_i+a_i$, hence $f'u_i=f_i$ and $\gamma(b_i)=a_i-a_{i_0}$. Now the equality $f=f'$ follows from the universal property of the direct sum $\bigoplus_{i\in I}G_i$.  
\end{construction}

\begin{example}[Pushouts]\label{ex:push}
    Since pushouts have a particularly nice form when compared to the other colimits, let us construct the pushout of the diagram $K \xleftarrow{f} G \xrightarrow{g} H$ in $T\HMod$ for $G$ non-empty. Pick $e \in G$. Then $\gG\big(K;f(e)\big) \xleftarrow{f} \gG\big(G;e\big) \xrightarrow{g} \gG\big(H;g(e)\big)$ is a diagram in $T\GMod$, whose pushout can be realised as
    \[\Big(P \coloneqq \gG\big(K;f(e)\big) \oplus  \gG\big(H;g(e)\big) / \langle f-g\rangle,\eta_K\colon K \to P, \eta_H\colon H \to P\Big).\]
    We claim that the heap of $T$-modules arising from $P$ is the pushout in the category of heaps of $T$-modules. In fact, if $K \xrightarrow{k} L \xleftarrow{h} H$ in $T\HMod$ are such that $k \circ f = h \circ g$, then $kf(e) = hg(e)$ and $\gG\big(K;f(e)\big) \xrightarrow{k} \gG\big(L;hg(e)\big) \xleftarrow{h} \gG\big(H;g(e)\big)$ are arrows in $T\GMod$ such that $k \circ f = h \circ g$, whence there exists a unique morphism $l \colon P \to L$ of pointed $T$-modules such that $l \circ \eta_K = k$ and $l \circ \eta_H = h$. This is also a morphism of heaps of $T$-modules and if $l'\colon P \to L$ is another such morphism satisfying the same property, then 
    \[l'\big((f(e),g(e))\big) = l'\eta_K\big(f(e)\big) = kf(e) = hg(e)\,,\]
    and hence $l'\colon P \to \gG(L;hg(e))$ is a morphism of pointed $T$-modules, which has to coincide with $l$ by uniqueness. Therefore, the functor $\Gg$ from Remark \ref{rem:TgrpsTHmods}, determined by retracting at a chosen element, creates pushouts.
\end{example}

\begin{example}[{$\star \sqcup \star$}]\label{ex:starcupstar}
    Recall that we denote by $\star$ the singleton $\{*\}$. It has a natural heap of $T$-modules structure given by $[*,*,*] = *$ and $t\triangleright_** = *$ for all $t \in T$. The (unique) associated group structure is the trivial group $\{0\}$. If $T$ admits an identity $1_T$, then the coproduct $\star \sqcup \star$ in $T\HMod$ can be realised as the heap of $T$-modules associated to the free pointed $T$-module $\rR(T)$ from \cref{cor:freeTgrp*}; that is, the set $\gG(T;1_T) \times \ZZ$ with heap of $T$-modules structure
    \[
    \begin{aligned}
    \big[(t,n),(t',n'),(t'',n'')\big] & = \big([t,t',t''],n-n'+n''\big) \\
    t \triangleright_{(t',m)} (t'',n) & = \big[t\cdot(t'',n),t\cdot(t',m),(t',m)\big] \\
    & = \Big(tt'' + (n-1)t - tt' - (m-1)t + t',n\Big) \\
    & = \Big(t \triangleright_{t'}t'' + (n-m)t, n\Big)
    \end{aligned}
    \]
    and canonical maps $\eta_1 \colon \star \to \rR(T), * \mapsto (1_T,0)$ and $\eta_2 \colon \star\to \rR(T), * \mapsto (1_T,1)$. This is easily checked directly by observing that $t \triangleright_{(1_T,0)}(s,n) = t \cdot (s,n)$, since $(1_T,0)$ is the zero of the ring $\rR(T)$, and therefore (see \eqref{eq:RTgen}) 
    \[
    \begin{aligned}
    (t,n) & = t \triangleright_{(1_T,0)} (1_T,1) + (n-1)(1_T,1) \\
    & = 
    \begin{cases} 
    \big[t \triangleright_{(1_T,0)} (1_T,1),(1_T,0),(1_T,1),(1_T,0),\ldots,(1_T,1)\big], & n-1 \geq 0, \\
    \big[t \triangleright_{(1_T,0)} (1_T,1),(1_T,1),(1_T,0),\ldots,(1_T,1),(1_T,0)\big], & n-1 < 0,
    \end{cases}
    \end{aligned}
    \]
    which means that any morphism of heaps of $T$-modules $\phi \colon \rR(T) \to M$ is uniquely and freely determined by $\phi(1_T,0)$ and $\phi(1_T,1)$, and that any morphism of heaps of $T$-modules $\psi \colon M \to \rR(T)$ such that $\{(1_T,0),(1_T,1)\}\subseteq \Ima(\psi)$ is automatically surjective.

    If $T$ does not have identity, then we carry on the same argument by replacing $\rR(T)$ with its Dorroh extension $\rR(T)_u$, which is unital. In particular, if $T = \varnothing$, then $\star \sqcup \star = \hH(\ZZ)$, the abelian heap associated with the group of integers. This is consistent with the fact that $\varnothing\HMod \cong \Ah$.
\end{example}

Our next result, \cref{thm:slices}, relies deeply on the categorical constructions we performed above, but it also provides a novel and easier way of realising them. It also supports once more the interpretation of heaps of $T$-modules as ``affine spaces over $T$''.

We first need a short technical lemma.

\begin{lemma}\label{lem:isotropic}
Let $T$ be a truss. Then the categories of heaps of $T$-modules and of isotropic heaps of $T_u$-modules are isomorphic. 
\end{lemma}

\begin{proof}
Let $M$ be a heap of $T$-modules. In view of \cref{prop:HtM}, this is equivalent to claim that we have a morphism of abelian heaps
\[\Delta_M \colon M \to \trs(T,\eE(M))\]
satisfying \ref{item:2.19a} and \ref{item:2.19b}. Since for every $m \in M$ we have that $\Delta_M(m) \colon T \to \eE(M)$ is a truss morphism and $\eE(M)$ is unital, the universal property of the Dorroh extension $T_u$ ensures that there is a unique unital extension $\tilde\Delta_M(m) \colon T_u \to \eE(M)$ such that $\tilde\Delta_M(m) \circ \jmath_T = \Delta_M(m)$ (see the end of \cref{ssec:TrandRing}). The uniqueness part allows us to define a heap morphism
\[\tilde\Delta_M \colon M \to \trs(T_u,\eE(M)),\]
which still satisfies \ref{item:2.19a} and \ref{item:2.19b} and converts $M$ into an isotropic heap of $T_u$-modules. In fact, since pre- and post-composing by heap morphisms is a heap morphism, it suffices to check that \ref{item:2.19a} and \ref{item:2.19b} are satisfied on elements $t \in T$ (which is obviously the case) and on $* = 1_{T_u}$ (which is the case because $\tilde\Delta_M(m)(*) = \id$ for every $m \in M$). Summing up, we can construct a functor 
\[\Phi \colon T\HMod\to T_u\HMod_{\mathbf{is}},\qquad  
(M,\triangleright) \mapsto (M,\tilde{\triangleright}), 
\]
(where $\tilde\triangleright$ is associated with $\tilde\Delta_M$ as in \cref{prop:HtM}) and which is the identity on morphisms. Concretely, $\tilde\triangleright$ is uniquely determined by
\[t ~ \tilde\triangleright_m ~ n = t \triangleright_mn \qquad \text{and} \qquad * ~ \tilde\triangleright_m ~ n = n\]
for all $t \in T$, $m,n \in M$.
In the opposite direction, let us consider the ``restriction of scalars'' functor 
\[
\Xi \colon T_u\HMod_{\mathbf{is}}\to T\HMod, \qquad (M,\triangleright)\mapsto (M,\triangleright\circ (\jmath_T\times \id_M\times \id_M)),
\]
where $\jmath_T:T\to T_u$ is a canonical embedding, and which acts as the identity of morphisms. One can easily check that $\Xi$ is a well-defined functor.

Now, $\Xi\circ \Phi=\id_{T\HMod}$ as if $(M,\triangleright)\in T\HMod$, then extending and restricting the scalars gives back the original heap of modules:
\[
(\Xi\circ \Phi) (M,\triangleright)=(M,\tilde{\triangleright}\circ (\jmath_T\times \id_M\times \id_M))=(M,\triangleright),
\]
by definition of $\tilde{\triangleright}$. Similarly $\Phi\circ\Xi=\id_{T_u\HMod_{\mathbf{is}}}$, as if $(M,\triangleright)\in T_u\HMod_{\mathbf{is}}$, then
\[
(\Phi\circ\Xi)(M,\triangleright) = \left(M,\widetilde{\triangleright\circ (\jmath_T\times \id_M\times \id_M)}\right) = (M,\triangleright),
\]
where the last equality follows from the fact that the unique way of extending the restricted heap of module structure $\triangleright\circ (\jmath_T\times \id_M\times \id_M)$ to $T_u$ is by letting $*$ act as the identity, which is already the case on $M$ since it is isotropic.
\end{proof}

Suppose that $T$ is a (possibly unital) truss. Recall that the \emph{slice category of (possibly unital) pointed $T$-modules over $\rR(T)$} has as objects the pairs $(G,G \to \rR(T))$ consisting of a (possibly unital) pointed $T$-module $G$ and a morphism of pointed $T$-modules $G \to \rR(T)$, and it has as morphisms those morphisms of pointed $T$-modules that make the obvious triangles commute. If we consider its full subcategory consisting of all $(G,\pi_G\colon G \to \rR(T))$ such that $\pi_G$ is surjective, then we may expand it harmlessly by allowing $\big(\{0\},0\colon \{0\} \to \rR(T)\big)$ to be the initial object. Denote the resulting category by $T\GMod\twoheaddownarrow \rR(T)$. It will be clear from the context if we are in the unital case or not.

\begin{theorem}\label{thm:slices}
    Let $T$ be a unital truss. The category of isotropic heaps of $T$-modules is equivalent to the category $T\GMod\twoheaddownarrow \rR(T)$, the full subcategory of slice category of pointed $T$-modules over $\rR(T)$ described above.
\end{theorem}

\begin{proof}
    The two initial objects, $\big(\{0\},0\colon \{0\} \to \rR(T)\big)$ in $T\GMod\twoheaddownarrow \rR(T)$ and $\varnothing$ in $T\HMod$, correspond to each other, so let us ignore them in what follows.
    
    Let $(G,\pi_G)$ be an object in $T\GMod\twoheaddownarrow \rR(T)$ and consider $\MMM(G) \coloneqq \{x \in G \mid \pi_G(x) = (1_T,1)\}$. It is a heap of $T$-modules with respect to
    \[
    \begin{gathered}
    [x,y,z] \coloneqq x-y+z \qquad \text{and} \qquad t \triangleright_x y \coloneqq t\cdot y -t\cdot x + x
    \end{gathered}
    \]
    for all $x,y,z \in \MMM(G)$ and $t \in T$. 
    Furthermore, since a morphism from $(G,\pi_G)$ to $(H,\pi_H)$ in $T\GMod\twoheaddownarrow \rR(T)$ is a morphism of pointed $\rR(T)$-modules $f \colon G \to H$ such that $\pi_H \circ f = \pi_G$, it is clear that we can (co)restrict $f$ to $\MMM(f) \colon \MMM(G) \to \MMM(H)$ and the latter is a morphism of heaps of $T$-modules. This construction induces a functor
    \[\MMM\colon T\GMod\twoheaddownarrow \rR(T) \lra T\HMod.\]

    In the opposite direction, let $M$ be a non-empty heap of $T$-modules and let $*_M\colon M \to \star$ be the associated (unique) morphism of heaps of $T$-modules to the terminal object. We can consider the coproduct $\star \sqcup M$ in $T\HMod$ and the corresponding pointed $\rR(T)$-module obtained by retracting at $*$, i.e.~$\Gg(\star \sqcup M; *)$, as in Remark \ref{rem:TgrpsTHmods}.
    By taking the coproduct of $\id_\star$ and $*_M$ and by keeping in mind Example \ref{ex:starcupstar}, we realise that $\GGG(M) \coloneqq \Gg(\star \sqcup M; *)$ comes endowed with a morphism of pointed $\rR(T)$-modules $\pi_{\GGG(M)} \colon \Gg(\star \sqcup M; *) \to \Gg(\star \sqcup \star; *) \cong \rR(T)$ which is surjective, since $\pi_{\GGG(M)}(*) = (1_T,0)$ and $\pi_{\GGG(M)}(m) = (1_T,1)$ for all $m \in M$. Thus, $\GGG(M)$ is an object in $T\GMod\twoheaddownarrow \rR(T)$. Furthermore, if $f \colon M \to N$ is a morphism of heaps of $T$-modules, then we have the morphism of heaps of $T$-modules $\id_\star \sqcup f \colon \star \sqcup M \to \star \sqcup N$ induced by the universal property of the coproduct and, since it maps $*$ to $*$,  it induces a morphism of pointed $\rR(T)$-modules
    $\GGG(f) \colon \GGG(M) \to \GGG(N)$. In this way the functor is obtained
    \[\GGG \colon T\HMod \lra T\GMod\twoheaddownarrow \rR(T).\]
    Now, observe that if $M$ is a heap of $T$-modules, then $M \subseteq \MMM\GGG(M)$ by construction of $\pi_{\GGG(M)}$. Denote by $\zeta_M$ this inclusion, which is the corestriction of the canonical morphism $\eta_M \colon M \to \star \sqcup M$.
    It is a morphism of heaps of $T$-modules, natural in $M$. 

    On the other hand, let $(G,\pi_G)$ be an object in $T\GMod\twoheaddownarrow \rR(T)$. Consider the obvious inclusion $\gamma_G \colon \MMM(G) \to G$ (which is a morphism in $T\HMod$ if we look at $G$ with its heap of $T$-modules structure $\Hh(G)$ from Remark \ref{rem:TgrpsTHmods}) and $\gamma_\star \colon \star \to \Hh(G), * \mapsto 0_G$. By the universal property of the coproduct, these induce a unique morphism of heaps of $T$-modules $\theta_G \colon \star \sqcup \MMM(G) \to \Hh(G)$ which, in turn, induces a morphism of pointed $\rR(T)$-modules
    \[\epsilon_G \colon \GGG\MMM(G) = \Gg(\star\sqcup \MMM(G);*) \lra \Gg(\Hh(G);0_G) = G.\]
    It is straightforward to check that $\theta_G$ is natural in $G$, thus $\epsilon_G$ is natural in $G$, too. 

    We leave to the reader to check that the resulting natural transformations $\zeta$ and $\epsilon$ satisfy the conditions to be unit and counit of an adjunction, respectively. 

    To conclude, we show that $\zeta$ and $\epsilon$ are, in fact, natural isomorphisms. Let us begin with $\zeta$. By keeping in mind \cref{con:coprod}, pick $e \in M$ and consider the isomorphism
    \begin{equation}\label{eq:isopsi}
    \star \sqcup M \cong \Hh\Big(\gG\big(M;e\big) \oplus \rR(T)\Big),
    \end{equation}
    uniquely determined by $* \mapsto (e,0_\rR)$ and $m \mapsto (m,1_\rR)$, for all $ m\in M$. The projection $\pi_{\GGG(M)}\colon \Gg(\star \sqcup M;*) \to \rR(T)$ from above coincides, up to the aforementioned isomorphism \eqref{eq:isopsi}, with the canonical projection on the second summand
    \[\gG\big(M;e\big) \oplus \rR(T) \lra \rR(T), \qquad (m,r) \lto r.\]
    This makes it clear that $\zeta_M(M) = \MMM\GGG(M)$ and hence that $\zeta$ is bijective.
    
    To prove that $\epsilon$ is a natural isomorphism, let $(G,\pi_G)$ be in $T\GMod\twoheaddownarrow \rR(T)$ and let $x \in G$ be such that $\pi_G(x) = (1_T,1)$. In view of \cref{con:coprod}, 
    $\star \sqcup \MMM(G)$ is isomorphic to $\Hh\Big(\gG\big(\MMM(G);x\big) \oplus \rR(T)\Big)$
    in such a way that $*$ corresponds to $(x,0_\rR)$, so that
    \[\GGG\MMM(G) = \gG\Big(\star \sqcup \MMM(G);*\Big) \cong \gG\big(\MMM(G);x\big) \oplus \rR(T).\]
    Furthermore, consider the split short exact sequence of pointed $\rR(T)$-modules
    \[0 \to \ker(\pi_G) \xrightarrow{j} G \xrightarrow{\pi_G} \rR(T) \to 0\]
    with chosen splitting $\bar{x}\colon \rR(T) \to G, r \mapsto r \cdot x$. Since $\MMM(G) = x + \ker(\pi_G)$, the automorphism $\tau_x^{0_G}\colon G \to G$ induces an isomorphism of pointed $\rR(T)$-modules $\gG\big(\MMM(G);x\big) \cong \ker(\pi_G)$. Therefore, we obtain an isomorphism of heaps of $T$-modules
    \[\star \sqcup \MMM(G) \to \Hh\Big(\gG\big(\MMM(G);x\big) \oplus \rR(T)\Big) \xrightarrow{\Hh(\tau_{x}^{0_G} \oplus \id_{\rR})} \Hh\Big(\ker(\pi_G) \oplus \rR(T)\Big) \xrightarrow{\Hh(j + \bar{x})} \Hh(G)\]
    such that: if we precompose it with $\eta_\star$, it maps $*$ to $(x,0_\rR)$, $(0_G,0_\rR)$ and, finally, to $0_G$; if we precompose it with $\eta_{\MMM(G)}$ then it maps every $y \in \MMM(G)$ to $(y,1_\rR)$, $(y-x,1_\rR)$ and, finally, to $(y-x)+1_\rR\cdot x = y$ (because unital pointed $T$-modules are unital $\rR(T)$-modules). By the uniqueness part of the universal property of the coproduct, this isomorphism of heaps of $T$-modules must be equal to $\theta_G$. As a consequence, $\epsilon_G$ is an isomorphism as well.
\end{proof}

\begin{corollary}\label{cor:slice}
    Let $T$ be a truss, not necessarily unital. Then the category of heaps of $T$-modules is equivalent to the full subcategory $T\GMod\twoheaddownarrow \rR(T)_u$ of the slice category of pointed $T$-modules over $\rR(T)_u$.
\end{corollary}

\begin{proof}
    For the sake of clarity and only in this proof and \cref{ex:empty}, let us denote by $R\Mod\twoheaddownarrow M$ the slice category of $R$-modules projecting onto $M$, together with $0$. In this way, the slice category of pointed unital $T_u$-modules over $\rR(T_u)$ would be $\rR(T_u)_1\Mod \twoheaddownarrow \rR(T_u)$, where the index $(-)_1$ highlights the fact that the modules are supposed to be unital. 
    For $T$ an arbitrary truss, \cref{lem:isotropic} entails that $T\HMod \cong T_u\HMod_{\mathbf{is}}$. Then, from \cref{thm:dummyP} and \cref{thm:slices} it follows that 
    \[T_u\HMod_{\mathbf{is}} \cong \rR(T_u)_1\Mod \twoheaddownarrow \rR(T_u).\]
    By applying further \cref{prop:Dorroh}, we get that
    \[\rR(T_u)_1\Mod \twoheaddownarrow \rR(T_u) \cong {\rR(T)_u}_1\Mod \twoheaddownarrow \rR(T)_u \cong \rR(T)\Mod \twoheaddownarrow \rR(T)_u,\]
    where the last equivalence is a well-known consequence of the universal property of the Dorroh extension for rings. Therefore,
    \[T\HMod \cong \rR(T)\Mod \twoheaddownarrow \rR(T)_u,\]
    as claimed, by \cref{thm:dummyP} again.
\end{proof}

\begin{example}\label{ex:empty}
    For $T = \varnothing$, \cref{ex:ahTHmod} and \cref{cor:slice} entails that
    \begin{equation*}
    \Ah = \varnothing\HMod \cong \rR(\varnothing)\Mod\twoheaddownarrow \rR(\varnothing)_u \cong \Ab\twoheaddownarrow\ZZ. \qedhere
    \end{equation*}
\end{example}

The description granted by \cref{thm:slices} and \cref{cor:slice} makes it easier, even if less ``elementary'', to describe some familiar categorical constructions in the category of heaps of $T$-modules. For example, if $(G,\pi_G)$ and $(H,\pi_H)$ are two heaps of $T$-modules, then their coproduct is simply $(G \oplus H, \pi_G + \pi_H)$ where the direct sum is taken in $\rR(T)\Mod$. Similarly, the coequalizer of two parallel morphisms $f,g \colon (G,\pi_G) \to (H,\pi_H)$ is $\big(H/\Ima(f-g),\tilde{\pi}_H\big)$ where $\tilde{\pi}_H \colon H/\Ima(f-g) \to \rR(T)$ is the factorisation of $\pi_H$ through the quotient.

\begin{remark}
    For the convenience of the reader, let us spend a few words more on the case of coproducts. Recall that we already saw an explicit realization of the coproduct in $T\HMod$ in \cref{con:coprod}. This is consistent with the one offered by \cref{thm:slices}. Let us show this for the coproduct of three heaps of $T$-modules $M,N,P$. The latter would be, according to \cref{thm:slices}, the preimage of $1_{\rR(T)}$ along the unique morphism 
    \[\pi \colon \GGG(M) \oplus \GGG(N) \oplus \GGG(P) \to \rR(T)\]
    induced by $\pi_{\GGG(M)}$, $\pi_{\GGG(N)}$ and $\pi_{\GGG(P)}$. As we did in the proof of \cref{thm:slices} itself, we may realize $(\GGG(M),\pi_{\GGG(M)})$ as $\gG(M;o_M) \oplus \rR(T)$ for some $o_M \in M$, with the projection on the second summand, and analogously for $N$ and $P$. Therefore, up to these realizations,
    \[
    \begin{gathered}
    \pi \colon \gG(M;o_M) \oplus \rR(T) \oplus \gG(N;o_N) \oplus \rR(T) \oplus \gG(P;o_P) \oplus \rR(T) \lra \rR(T), \\ 
    (m,x,n,y,p,z) \lto x + y + z,
    \end{gathered}
    \]
    and hence
    \[
    \begin{aligned}
    M \sqcup N \sqcup P & = \MMM\big(\gG(M;o_M) \oplus \rR(T) \oplus \gG(N;o_N) \oplus \rR(T) \oplus \gG(P;o_P) \oplus \rR(T)\big) \\
    & \cong \Hh\big(\gG(M;o_M) \oplus \gG(N;o_N) \oplus \gG(P;o_P) \oplus \rR(T) \oplus \rR(T)\big)
    \end{aligned}
    \]
    as prescribed by \cref{con:coprod}.
\end{remark}

\section{Exact sequences of heaps of modules}\label{sec:five}

In this section we are concerned with the description of (short) exact sequences of abelian heaps and heaps of $T$-modules. 
In view of \cref{ex:ahTHmod}, it is enough to deal with exact sequences of heaps of $T$-modules.

\subsection{Exact sequences of heaps of modules: ring-theoretic approach}\label{sssec:exheap}

Let $T$ be a truss. By adapting the definition presented in \cite[\S6.1]{BrzRyb:fun}, let 
\begin{equation}\label{eq:ses}
M \xrightarrow{f} N\xrightarrow{g} P
\end{equation}
be a sequence of heaps of $T$-modules and morphisms of heaps of $T$-modules. We may suppose, without loss of generality, that $P \neq \varnothing$. 

\begin{definition}\label{def:ses}
We say that the sequence \eqref{eq:ses} is \textit{exact in $N$ at $e$} if there exists $e\in g(N)$ such that $f(M)=g^{-1}(e)$.
\end{definition}

\begin{remark}\label{rem:sesRTmods}
\cref{def:ses} can also be interpreted as follows. Suppose that \eqref{eq:ses} is exact in $N$ at $o_P \in g(N) \subseteq P$. If we fix an element $o_N\in N$ such that $g(o_N)=o_P$, then the induced pointed $T$-module structures (see \cref{rem:TgrpsTHmods}) on $N$ and $P$ are such that $o_N$ and $o_P$ are the neutral elements of the corresponding $\rR(T)$-modules (see \cref{thm:dummyP}) and $g$ is a morphism of $\rR(T)$-modules. If moreover we fix an element $o_M \in M$ such that $f(o_M) = o_N$ and we consider the induced $\rR(T)$-module structure on $M$, then $f$ will be also a morphism of $\rR(T)$-modules. In this context, the sequence $\Gg(M;o_M) \xrightarrow{f} \Gg(N;o_N) \xrightarrow{g} \Gg(P;o_P)$ is an exact sequence of $\rR(T)$-modules. The converse clearly holds: if there are $o_M \in M$, $o_N \in N$, $o_P \in P$ such that the sequence of $\rR(T)$-modules $\Gg(M;o_M) \xrightarrow{f} \Gg(N;o_N) \xrightarrow{g} \Gg(P;o_P)$ is exact, then \eqref{eq:ses} is exact in $N$ at $o_P$.
\end{remark}

By a slight abuse of notation, we may often simply refer to the $\rR(T)$-module $M$ instead of $\Gg(M;o_M)$, once $o_M$ has been chosen. Since the underlying set is the same, this allows us to lighten the notation and increase the readability.

\begin{definition}
A sequence of morphisms of heaps of $T$-modules
$$\cdots \to M_{n-1}\overset{f_n}\to M_n\overset{f_{n+1}}\to M_{n+1}\to \dots$$
is \textit{exact} if it is exact in $M_n$ for all $n$. 
\end{definition}

\begin{remark}
In view of \cref{rem:sesRTmods}, it follows that for every $n$ we can find $\rR(T)$-module structures on $M_{n-1}$, $M_n$, and $M_{n+1}$ such that $M_{n-1}\overset{f_n}\to M_n\overset{f_{n+1}}\to M_{n+1}$ is an exact sequence of $\rR(T)$-modules. However, it is easy to see that if we consider chains with more than two morphisms, then the definition does not necessarily allow us to choose group structures on all heaps of $T$-modules such that the sequence is an exact sequence of $\rR(T)$-modules.
\end{remark}

A bit more generally than \cref{rem:sesRTmods}, but at the price of changing the morphisms involved, we have the following result. 

\begin{proposition}\label{Prop:exofheap}
Let $M \overset{f}\to N \overset{g}\to P$ be a sequence of heap of $T$-modules morphisms. Suppose that the structures on $M$, $N$, and $P$ are induced by some corresponding $\rR(T)$-module structures $\Gg(M;o_M)$, $\Gg(N;o_N)$, $\Gg(P;o_P)$, such that $f=\alpha+h$, $g=\beta+k$, where $\alpha$ and $\beta$ are morphisms of $\rR(T)$-modules, $h = f(o_M)\in N$, and $k = g(o_N)\in P$. Then  $M \overset{f}\to N \overset{g}\to P$ is an exact sequence of heaps of $T$-modules if and only if $M \overset{\alpha}\to N \overset{\beta}\to P$ is an exact sequence of $\rR(T)$-modules.
\end{proposition} 

\begin{proof}
Suppose that $M\overset{f}\to N\overset{g}\to P$ is exact in $N$ at $e \in g(N)$. That is, $f(M)=g^{-1}(e)$. It follows that $e = gf(m) = \beta\alpha(m)+\beta(h)+k$ for all $m \in M$ and hence $\beta\alpha = o_P$ and $\beta(h)+k=e$. Remark that $h + \ker(\beta) \subseteq g^{-1}(e)$, but also the reverse inclusion holds: if $\beta(h) + k = e = g(n) = \beta(n) + k$, then $n \in h + \ker(\beta)$.
Since $f(M) = \alpha(M) + h$, it follows that $f(M) = g^{-1}(e)$ entails that $\Ima(\alpha) = \ker(\beta)$, and hence that $M \overset{\alpha}\to N \overset{\beta}\to P$ is exact in $N$.

Conversely, take $e \coloneqq \beta(h)+k$. Then $gf=e$ and so $f(M)\subseteq g^{-1}(e)$. In the opposite direction, if $g(x)=e$ then $\beta(x)+k=\beta(h)+k$ and hence $x\in h + \ker(\beta) = h+\Ima(\alpha)=f(M)$. 
\end{proof}

\begin{remark}
 If $M\overset{f}\to N\overset{g}\to P$ is an exact sequence of heaps of $T$-modules, then all the ``perturbed" sequences $M\xrightarrow{f+n} N\xrightarrow{g+p} P$, for $n\in N$ and $p\in P$, are exact sequences of heaps of $T$-modules, where $+$ is with respect to a certain choice of $o_M \in M$, $o_N \in N$ and $o_P \in P$.
\end{remark}

In this framework, we may say that a sequence $M \xrightarrow{f} N \xrightarrow{g} P$ is a \emph{short exact sequence} if it is exact in the above sense and, moreover, $f$ is injective and $g$ is surjective.

\subsection{Exact sequences of heaps of modules: Barr's approach}\label{ssec:Barr}
For the sake of completeness, in this section we would like to rely on the fact that $T\HMod$, being an algebraic category (i.e.\ a category of models for an algebraic theory in the sense of, e.g., \cite[\S3.9]{Bor-Hand}), is an \emph{exact category} in the sense of Barr \cite{Barr} (see also \cite[\S2.6]{Bor-Hand} or \cite[A.5]{Bor-Bou:mal} for a more concise introduction). Hence, we may discuss short exact sequences of heaps of $T$-modules in the sense of, e.g., \cite[\S2.3]{Bor-Hand}. 
In this framework, a \emph{short exact sequence (in the sense of Barr)} of heaps of $T$-modules is a sequence
\[
\xymatrix{
M \ar@<+0.4ex>[r]^-{f} \ar@<-0.4ex>[r]_-{g} & N \ar[r]^-{h} & P
}
\]
such that $(f,g)$ is the kernel pair of $h$ and $h$ is the coequalizer of $(f,g)$. 

In a way similar to the case of abelian categories \cite[Proposition 2.3.2]{Bor-Hand}, there is a close, but not completely direct, correspondence between short exact sequences of heaps of $T$-modules in the sense of Barr and in the ring-theoretic sense.

\begin{theorem}\label{thm:exex}
Let    
 \begin{equation}\label{diag.fork}
    \xymatrix{
M \ar@<+0.4ex>[r]^-{f} \ar@<-0.4ex>[r]_-{g} & N \ar[r]^-{h} & P
}    
    \end{equation}
    be a diagram of heaps of $T$-modules. Then the following statements are equivalent:
    \begin{zlist}
        \item\label{item:4.23.1} The sequence \eqref{diag.fork} is exact in the sense of Barr.
        \item\label{item:4.23.2} There exists $o\in N$ (hence for all $o\in N$) such that, for $h_o \colon N\times N \to P$ given by $h_o(a,b) \coloneqq h\big([a,b,o]\big) = \big[h(a),h(b),h(o)\big]$ for all $a,b\in N$, the sequence
\begin{equation}\label{seq.ex}
\xymatrix{M \ar[r]^-{(f,g)} & N\times N \ar[r]^-{h_o} & P},   
\end{equation}
is a short exact sequence of heaps of $T$-modules, exact in $N \times N$ at $h(o)$.
    \end{zlist}
\end{theorem}

\begin{proof}
    Assume that assertion (\ref{item:4.23.1}) holds and take any $o\in N$. Define $h_o$ as in assertion (\ref{item:4.23.2}). Since $h$ and $[-,-,-]$ are morphisms of heaps of $T$-modules (see \cref{ex:prods}), $h_o$ is a morphism of heaps of $T$-modules, too, and since \eqref{diag.fork} is a fork of homomorphisms of heaps of $T$-modules, we obtain that $h_o \circ (f,g) \equiv h(o)$ and so $\Ima (f,g)\subseteq h_o^{-1}\big(h(o)\big)$. Conversely, if $(a,b) \in h_o^{-1}\big(h(o)\big)$, then $[h(a),h(b),h(o)] =h(o)$ and so $h(a)=h(b)$. Thus, $(a,b)$ is an element of the kernel pair of $h$. By assumption, $M$ is such a kernel pair and $f$ and $g$ are its structural maps, which implies that $(a,b)\in \Ima (f,g)$. Therefore, $\Ima (f,g) = h_o^{-1}\big(h(o)\big)$. Since $h$ is the structural map of a coequalizer, it is an epimorphism and hence surjective (because $T\HMod$ is an algebraic category and so \cite[Corollary 3.5.3]{Bor-Hand} applies). Therefore, for any $p\in P$ there exists $n\in N$, such that $p=h(n) = [h(n),h(o),h(o)] = h_o(n,o)$. This means that $h_o$ is a surjective homomorphism. The map $(f, g)$ is injective by the realization of the kernel pair as $\{(m,n) \in N \times N \mid h(m)=h(n)\}$. In summary, the assertion (\ref{item:4.23.2}) holds.

    Now assume that assertion (\ref{item:4.23.2}) holds. First we show that \eqref{diag.fork} is a kernel pair. Since, by hypothesis, $\Ima (f, g)\subseteq h_o^{-1}\big(h(o)\big)$, we have that 
    \[
    h(o) = \big(h_o\circ (f, g)\big)(m) = \big[(h\circ f) (m), (h\circ g)(m) ,h(o
    )\big],
    \]
    for all $m\in M$, which implies that $h\circ f = h \circ g$. Hence \eqref{diag.fork} is a fork.  Let $\alpha,\beta: Q\to N$ be morphisms of heaps of $T$-modules such that $h \circ \alpha = h \circ \beta$.  Then $h_o \circ (\alpha, \beta) = h(o)$, and so, by exactness of \eqref{seq.ex} in $N\times N$ at $h(o)$, for all $q\in Q$, there exists $m_q\in M$, such that $\alpha(q) = f(m_q)$ and $\beta(q) = g(m_q)$. Since $(f, g)$ is injective, for a given $q$, $m_q$ is unique. This allows one to define a homomorphism of heaps of $T$-modules $\gamma: Q\to M, q\mapsto m_q$. By construction, this is the unique homomorphism such that $\alpha = f \circ \gamma$ and $\beta =g \circ \gamma$. Therefore, \eqref{diag.fork} is a kernel pair.
    To check that it is coequalizer, take any heap of $T$-modules morphism $\alpha: N\to Q$ such that $\alpha \circ f = \alpha \circ g$. Set $\alpha_{o} \colon N\times N \to Q, (a,b)\mapsto \big[\alpha(a),\alpha(b), \alpha(o)\big]$. Then $\alpha_{o} \circ (f, g) = \alpha(o)$ and hence there exists unique $\beta \colon P\to Q$ such that $\alpha_{o} = \beta \circ h_{o}$. Note that, for all $n\in N$,
    \[
    (\beta\circ h)(n) = \beta\big([h(n), h(o), h(o)]\big) = (\beta \circ h_o)(n,o) = \alpha_{o}(n,o) = [\alpha(n),\alpha(o),\alpha(o)] = \alpha(n),
    \]
    i.e., $\alpha = \beta \circ h$. Finally, by surjectivity of $h_o$, for all $p\in P$ there exists $(a,b)\in N\times N$ such that $p = h_o(a,b) = \big[h(a), h(b), h(o)\big] = h\big([a,b,o]\big)$ and hence $h$ is surjective. This completes the proof that \eqref{diag.fork} is a short exact sequence in the sense of Barr.
\end{proof}

\section*{Acknowledgments}
The research of Simion Breaz is supported by a grant of the Ministry of Research, Innovation and Digitization, CNCS/CCCDI--UEFISCDI, project number PN-III-P4-ID-PCE-2020-0454, within PNCDI III.

The research of Tomasz Brzezi\'nski is partially supported by the National Science Centre, Poland, grant no. 2019/35/B/ST1/01115. 

The research of Bernard Rybo\l owicz is supported by the EPSRC grant EP/V008129/1.

This paper was written while Paolo Saracco was a Charg\'e de Recherches of the Fonds de la Recherche Scientifique - FNRS and a member of the ``National Group for Algebraic and Geometric Structures and their Applications'' (GNSAGA-INdAM).

\end{document}